\documentclass[12pt,reqno]{amsart}

\usepackage{amsmath, amsthm, amssymb, amscd, amsfonts}
\usepackage{fontsmpl}
\usepackage{fontsmpl,layout}
\usepackage{color}
\usepackage[all]{xy}             
 \usepackage[active]{srcltx}   

\usepackage{t1enc}
\usepackage[T1]{fontenc}
\usepackage{tikz}
\usepackage{mathrsfs} 
\usepackage{pb-diagram}
\usepackage{array}

 at 10truept
 at 9truept

\newcommand{\subu}[2]{{#1}_{\raise-2pt\hbox{$ \scriptstyle #2 $}}}
\newcommand{\subd}[3]{{#1}_{\raise-2pt\hbox{$ \scriptstyle #2 #3 $}}}

\newtheorem{mtheorem}{Theorem}
\newtheorem{lema}{Lemma}[section]
\newtheorem{theorem}[lema]{Theorem}
\newtheorem{cor}[lema]{Corollary}

\newtheorem{prop}[lema]{Proposition}

\theoremstyle{definition}
\newtheorem{definition}[lema]{Definition}
\newtheorem{exa}[lema]{Example}

\newtheorem{rmk}[lema]{Remark}

\theoremstyle{remark}
\newtheorem{obs}[lema]{Remark}

\newcommand{\q}{\mathcal{Q}}

\newcommand\id{\operatorname{id}}

\newcommand\co{\operatorname{co}}

\newcommand\rk{\operatorname{rk}}
\newcommand\ord{\operatorname{ord}}
\newcommand\cop{\operatorname{cop}}

\newcommand\Hom{\operatorname{Hom}}

\newcommand\Ker{\operatorname{Ker}}
\newcommand\Img{\operatorname{Im}}

\newcommand\op{\operatorname{op}}

\newcommand{\eps}{\varepsilon}
\newcommand{\ot}{\otimes}
\newcommand{\com}{\Delta}

\newcommand\Supp{\operatorname{Supp}}
\newcommand\Res{\operatorname{Res}}
\newcommand\res{\operatorname{res}}
\newcommand\diag{\operatorname{diag}}
\newcommand\Fr{\operatorname{Fr}}
\newcommand{\g}{\mathfrak g}

\newcommand{\W}{{\mathcal W}}

\newcommand{\C}{{\mathcal C}}
\newcommand{\D}{{\mathcal D}}

\newcommand{\Oc}{{\mathcal O}}

\newcommand\Lie{\operatorname{Lie}}

\def\NN{\mathbb{N}}
\def\ZZ {\mathbb{Z}}
\def\QQ{\mathbb{Q}}
\def\CC{\mathbb{C}}

\def\AA{\mathbb{A}}
\def\JJ{\mathcal{J}}
\def\II{\mathcal{I}}
\def\OO{\mathcal{O}}
\def\SS{\mathcal{S}}

\def\Tt{\mathbb{T}}

\def\D{\mathcal{D}}
\def\E{\mathcal{E}}

\def\lieg{\mathfrak{g}}

\def\lieh{\mathfrak{h}}

\def\liep{\mathfrak{p}}

\def\QEM{U_q(\lieg, M)}
\def\QEMphi{U_q^{\varphi}(\lieg, M)}
\def\QEPbmasphi{\check{U}_q^{\varphi}(\bgot_{+})}
\def\QEPbmenphi{\check{U}_q^{\varphi}(\bgot_{-})}
\def\QEbmasphi{U_q^{\varphi}(\bgot_{+})}
\def\QEbmenphi{U_q^{\varphi}(\bgot_{-})}
\def\QEPphi{U_q^{\varphi}(\lieg, P)}
\def\QEphicheck{\check{U}_q^{\varphi}(\lieg)}
\def\QEcheck{\check{U}_q(\lieg)}
\def\QEQphi{U_q^{\varphi}(\lieg, Q)}

\def\Ga{\Gamma}

\def\OeP{\OO_{\epsilon}(P)}

\def\Oe{\OO_{\epsilon}(G)}
\def\Oq{\OO_{q}(G)}

\def\qe{{\bf u}_{\epsilon}(\lieg)}

\def\bgot{\mathfrak{b}}

\def\qep{{\bf u}_{\epsilon}(\liep)}

\def\pf{\begin{proof}}
\def\epf{\end{proof}}

\theoremstyle{plain}
\def\qephi{\mathbf{u}_\epsilon^\varphi(\mathfrak{g})}
\def\qephip{\mathbf{u}^\varphi_\epsilon(\mathfrak{p})}
\def\Gphige{\Gamma_\epsilon^\varphi(\mathfrak{g})}
\def\Gephip{\Gamma_\epsilon^\varphi(\mathfrak{p})}

\def\Gphig{\Gamma^\varphi(\mathfrak{g})}
\def\Gphibp{\Gamma^\varphi(\mathfrak{b}_{+})}
\def\Gphibm{\Gamma^\varphi(\mathfrak{b}_{-})}
\def\h{\mathfrak{h}}
\def\p{\mathfrak{p}}
\def\QEphi{U_q^\varphi(\mathfrak{g})}

\def\QUphie{U_\epsilon^\varphi}

\def\Rqphi{R_q^\varphi}
\def\OphieeQ{\OO_\epsilon^\varphi(G)_{\QQ_{(\epsilon)}}}
\def\Oephi{\OO_\epsilon^\varphi(G)}
\def\OephiP{\OO_\epsilon^\varphi(P)}
\def\Ophieq{\OO_q^\varphi(G)}
\def\QUhatphi{\check{U}^\varphi_q}

\def\QUhatphie{\check{U}^\varphi_\epsilon}
\def\QUhatphieR{\widehat{U}^\varphi_\epsilon}

\begin{document}





\title[Twisted quantum subgroups]
{Quantum subgroups of simple twisted quantum groups at roots of one}

\author{Gast\'on A. Garc\'ia \and Javier A. Guti\'errez}

\thanks{2010 Mathematics Subject Classification: 81R50, 17B37, 20G42, 16W30, 16W35. \\
\textit{Keywords: twisted quantum function algebras, quantum subgroups, Hopf algebra quotients} \\
This work was partially supported by
ANPCyT-Foncyt, CONICET, Secyt (UNC)}

\address{\noindent G. A. G. : Departamento de Matem\'atica, Facultad de Ciencias Exactas,
Universidad Nacional de La Plata. CONICET. Casilla de Correo 172, 1900
La Plata, Argentina.}

\email{ggarcia@mate.unlp.edu.ar}

\address{\noindent J. A. G. : FaMAF-CIEM (CONICET), Universidad Nacional de C\'ordoba. 
Medina Allende s/n, Ciudad Universitaria, 5000 C\'ordoba. Argentina.}

\email{puiguti@gmail.com }

\begin{abstract}
Let $G$ be a connected, simply connected simple complex algebraic group and let $\epsilon$ be a primitive
$\ell$th root of unity with $\ell$
odd and coprime with $3$ if $G$ is of type $G_{2}$.
We determine all  Hopf algebra quotients 
of the twisted multiparameter quantum function
algebra $\OO_{\epsilon}^{\varphi}(G)$ introduced by Costantini and Varagnolo. This extends  
the results of Andruskiewitsch and the first author, where the untwisted case is treated.
\end{abstract}


\maketitle

\section{Introduction}
Let $G$ be a connected, simply connected complex algebraic group.
In these notes we determine all Hopf algebra quotients of the twisted multiparameter 
quantum function algebra $\Oephi$
introduced by Costantini and Varagnolo in \cite{cv1}, where $\epsilon$ is a 
primitive $\ell$th root of unity with $\ell$
odd and if $G$ is of type $G_{2}$ further $\epsilon$ is coprime with $3$. 
The dual notion of this 
 was introduced 
by Reshetikhin \cite{R} to produce multiparameter quantum enveloping 
algebras of $\lieg=\Lie(G)$, see also \cite{Su}. 
It is constructed as a 
twist deformation of the topological Hopf algebra $U_{\hbar}(\lieg)$ over $\CC[[\hbar]]$, where the
twist only involves elements of a fixed Cartan subalgebra $\lieh$ of $\lieg$. In the dual function algebra, 
this deformation corresponds to a skew endomorphism $\varphi$ on the weight lattice of $\g$. 
When $\varphi=0$, one 
recovers the standard quantum function algebra on $G$ and the resuIts on this 
paper reproduce the classification 
obtained in \cite{AG}.

It turns out that $\Oephi$ is a $2$-cocycle deformation of 
$\Oe$, see Lemma \ref{lem:def2cocicloOe}. For this reason, we call $\Oephi$ the \textit{twisted quantum 
function algebra} over $G$; heuristically it should correspond to the function algebra over the 
\textit{twisted} quantum group $G^{\varphi}_{\eps}$.
This is not an isolated example. 
The relation between multiparameter quantum function algebras and $2$-cocycle deformations has 
been explained for particular instances
of quantum groups; see for example \cite{Ma}, \cite{Tk2}, \cite{AST}, \cite{hlt}.
In general, multiparameter quantum groups were intensively studied. 
They appeared first in the work of Manin \cite{Ma} and
were subsequently treated by different authors, among them \cite{AE,BW, CM,DPW,H,hlt,HPR,ls,OY,R,Tk}.

An important problem in the theory of quantum groups is the determination of the\linebreak general properties that 
a quantum group should have, since up to date there is no\linebreak axiomatic definition of an \textit{algebraic} quantum 
group.
In this sense, the 
description of all possible Hopf algebra 
quotients of the quantum function algebra, seen as algebras of functions over 
\textit{quantum subgroups}, 
of the known examples would give some insight on the structure of the quantum group. This can be viewed as 
the quantum version of the classical problem of studying subgroups of a simple algebraic group. This 
is an actual area of research since the description by P. Podle\'s \cite{podles} of the compact  quantum subgroups 
of Woronowicz's quantum groups $SU_{q}(2)$ and $SO_{q}(3)$ for $q\in [-1,1]\setminus 0$. 
Besides the result of Podle\'s, 
the main contributions are
\begin{itemize}
 \item[$\triangleright$] the description of the finite quantum subgroups of $GL_{q}(n)$ and $SL_{q}(n)$ for 
 $q$ an odd root of unity
 by M\"uller \cite{Mu};
 \item[$\triangleright$] the classification in \cite{AG} of the quantum subgroups of $G_{q}$,  
 for $G$ a connected, simply connected,
 complex simple algebraic group $G$, with $q$ a primitive
$\ell$th root of unity with $\ell$
odd and coprime with $3$ if $G$ is of type $G_{2}$.
  \item[$\triangleright$] the description  in \cite{G} of the  quantum subgroups of the two-parameter 
  deformation $GL_{\alpha,\beta}(n)$ for $\alpha^{-1}\beta $ a primitive  root of unity of odd order;
\item[$\triangleright$]  the compact quantum subgroups of $SO_{-1}(3) $ were determined by Banica and Bichon \cite{BB};
   \item[$\triangleright$] the study of the quantum subgroups of $SU_{q}(2)$ for $q=-1$ in \cite{BN} and for $q\neq -1$ in
   \cite{FST}.
\item[$\triangleright$] the description by  Bichon and Dubois-Violette \cite{BD} of the 
 compact quantum subgroups of  the half-liberated orthogonal quantum groups $O_{n}^{*}$
from \cite{BS}. 
  \item[$\triangleright$] the classification of the quantum subgroups of $SU_{-1}(3)$ by Bichon and Yuncken \cite{BY}.
 \end{itemize}

 As the reader might have noticed, the problem splits into the algebraic case and the compact case.	The latter 
 is reduced mainly to the case when $q=-1$. In this paper, we study the algebraic case, hence we will assume
 that $q$ is a primitive root of unity of odd order. The main result reads 
 
\begin{mtheorem}\label{teo:biyeccion}
There is a bijection between
\begin{enumerate}
\item[$(a)$] Hopf algebra quotients $q: \Oephi \to A$.
\item[$(b)$] Twisted subgroup data up to equivalence.
\end{enumerate}
\end{mtheorem}

For the definition of the twisted subgroup data see Definition \ref{def:twistedsdata}. We prove Theorem 
\ref{teo:biyeccion} in Section \ref{sec:qsubgroups} through Theorems \ref{thm:mainteo1}, \ref{thm:oephiisodualnfl}
and \ref{thm:classification}. We use the strategy developed in \cite{AG} for the untwisted case, where the
Hopf algebra quotients are constructed using commutative diagrams whose rows are central 
extension of Hopf algebras. Since $\Oephi$ is a $2$-cocycle deformation of $\Oe$, several steps
of the construction can be carried out without much effort. On the other hand, special attention 
has to be paid for certain constructions. To describe them we use the study of $\Oephi$
carried out by Costantini and Varagnolo in \cite{cv1} and \cite{cv2}, which is in turn a generalization
of \cite{dl}.

As a consequence of Theorem 
\ref{teo:biyeccion} one would expect 
 the construction of new examples of finite-dimensional Hopf algebras
with different properties which might not be necessarily $2$-cocycle deformation of 
Hopf algebra quotients of $\Oe$. These examples 
are given by central exact sequences of Hopf algebras. 
We hope that, using similar methods to those of \c Stefan \cite{St}, who characterized Hopf algebras
with certain properties as quotients of the quantum function algebra $\Oc_{q}(SL(2))$ 
and used this to understand the structure of
low-dimensional Hopf algebras, see \cite{N} as well, these examples might help to understand better
the classification problem for finite-dimensional Hopf algebras over the complex numbers.

The paper is organized as follows. In Section \ref{sec:twistedqg} we recall the
definition and general properties of the twisted quantum enveloping algebra 
$\QEphi$, its divided power algebra, the twisted quantum function algebra 
$\OO_{q}^{\varphi}(G)$ and their specializations at roots of unity, and we show that 
$\Oephi$ is a $2$-cocycle deformation
of $\Oe$. 
In Section \ref{sec:mflk} we describe the twisted Frobenius-Lusztig kernels
$\qephi$ and all the Hopf algebra quotients of $\qephi^{*}$. 
We also prove that $\qephi$ is a twist deformation of
$\qe$. Finally, in Section \ref{sec:qsubgroups} we prove the main theorem.

\subsection*{Conventions and Preliminaries}
Our references for the theory of Hopf algebras are \cite{Mo}, \cite{Ra}. 
We use standard notation for Hopf algebras; 
the comultiplication, counit and antipode are denoted by $\com$, $\eps$ and $\SS$,
respectively. Let $\Bbbk$ be a field.
 The set of group-like elements of a
coalgebra $C$ is denoted by $G(C)$. We also denote by $C^{+} = \Ker
\eps$ the augmentation ideal of $C$. 
Let $H$ be a Hopf algebra. $H^{\op}$  denotes 
the Hopf algebra with the same coalgebra structure but opposite multiplication
and $H^{\cop}$  denotes 
the Hopf algebra with the same algebra structure but opposite comultiplication. 
Let $g,h\in G(H)$, the set of 
$(g,h)$-primitive elements is given by 
$P_{g,h}(H)=\{x\in H:\ \com(x)= x\ot g + h\ot x\}$.
We call $P_{1,1}(H)=P(H)$ the set of primitive elements.

Recall that a convolution invertible linear map
$\sigma $ in $\Hom_{\Bbbk}(H\ot H, \Bbbk)$
is a
\textit{normalized multiplicative 2-cocycle} if
\begin{equation}\label{eq:concocycle}
\sigma(b_{(1)},c_{(1)})\sigma(a,b_{(2)}c_{(2)}) =
\sigma(a_{(1)},b_{(1)})\sigma(a_{(2)}b_{(2)},c)  
\end{equation}
and $\sigma (a,1) = \eps(a) = \sigma(1,a)$ for all $a,b,c \in H$,
see \cite[Sec. 7.1]{Mo}.
In particular, the inverse of $\sigma $ is given by
$\sigma^{-1}(a,b) = \sigma(\SS(a),b)$ for all $a,b\in H$.
Using a $2$-cocycle $\sigma$ it is possible to define
a new algebra structure on $H$ by deforming the multiplication,
which we denote by $ H_{\sigma} $. Moreover, $H_{\sigma}$ is indeed
a Hopf algebra with
$H = H_{\sigma}$ as coalgebras,
deformed multiplication
$m_{\sigma} = \sigma * m * \sigma^{-1} : H \ot H \to H$
given by $$m_{\sigma}(a,b) = a\cdot_{\sigma}b = \sigma(a_{(1)},b_{(1)})a_{(2)}b_{(2)}
\sigma^{-1}(a_{(3)},b_{(3)})\qquad\text{ for all }a,b\in H,$$
if $Q^\sigma=\sigma(Id\ot\SS)$ and $Q^{\sigma^{-1}}=\sigma^{-1}(\SS\ot Id )$ the antipode
$\SS_{\sigma} = Q^\sigma\ast\SS\ast Q^{\sigma^{-1}} : H \to H$ given by (see \cite{doi} for details)
$$\SS_{\sigma}(a)=\sigma(a_{(1)},\SS(a_{(2)}))\SS(a_{(3)})
\sigma^{-1}(\SS(a_{(4)}),a_{(5)})\qquad\text{ for all }a\in H.$$

\begin{obs}\label{obs:cocycleproy}
Let $H$ be a Hopf algebra, $I$ a Hopf ideal, 
$A=H/I$ and $\pi: H \to A$ the canonical map. Clearly, any $2$-cocycle on $A$ can be lifted through $\pi$
to a $2$-cocycle on $H$. Let 
$\sigma: H\otimes H\to \Bbbk$ a normalized multiplicative $2$-cocycle on $H$ such that
$\sigma|_{I\ot H + H\ot I} = 0$. Then the map $\hat{\sigma}: A\ot A \to \Bbbk$
given by $\hat{\sigma}(\pi(h),\pi(k)) = \sigma(h,k)$ for all $h,k \in H$ defines a  
 normalized multiplicative $2$-cocycle on $A$ and the induced map 
 $\pi_{\sigma}: H_{\sigma}\to A_{\hat{\sigma}}$ is a Hopf algebra map. In particular, if $B$ is 
 a central Hopf subalgebra of $H$ such that 
$\sigma|_{I\ot H + H\ot I} = 0$ with $I=B^{+}H$, then  
the formula above defines a $2$-cocycle on $A = H/ B^{+}H$. 
\end{obs}

Let $H$ be a Hopf algebra and $J\in H\ot H$ an invertible element. We say that
$J$ is a \textit{normalized twist} if
$$
(\com\ot \id)(J)(J\ot 1)  =  (\id\ot\com)(J)(1\ot J)\quad\text{and}\quad (\eps\ot \id)(J) =  1 = (\id\ot\eps)(J).
$$
Given a twist $J$ for $H$, one can define a new Hopf algebra $H^{J}$ with the same
algebra structure and counit as $H$, but different comultiplication and antipode 
$$
\com^{J}(h)  =  J^{-1}\com(h) J,\qquad \SS^{J}(h) = Q_{J}^{-1}\SS(h)Q_{J},
$$
for all $h \in H$, where we denote $J=J^{(1)}\ot J^{(2)}$ and $Q_{J} = \SS(J^{(1)})J^{(2)}$.
We say that $H^{J}$ is a twist deformation of $H$.

The notion of 2-cocycle and twist are dual of each other. If $H$ is finite-dimensional,
then $J$ is a twist for $H$ if and only if $J^{*}$ is a $2$-cocycle on $H^{*}$.

\begin{definition}\label{def:hopf-pairing}
A {\it Hopf pairing} between two Hopf algebras $U$ and $H$ over a
ring $R$ is a bilinear form $b: H \times U \to
R$ such that, for all $u,\ v \in U$ and $f,\ h \in H$,
\begin{align*}
& (i)\qquad  b(h,uv) = b(h_{(1)},u)b(h_{(2)},v);\qquad & (iii)&\qquad
b(1,u) =
\eps(u);\\
&(ii)\qquad b(fh,u) = b(f,u_{(1)})b(h,u_{(2)}); \qquad & (iv)& \qquad
b(h,1) = \eps(h).
\end{align*}
\end{definition}

\noindent It follows that $b(h,\SS(u)) = b(\SS(h),u)$ for all $u \in
U$, $h \in H$. Given a Hopf pairing, one has Hopf algebra maps $U
\to H^{\circ}$ and $H \to U^{\circ}$, where $H^{\circ}$ and
$U^{\circ}$ are the Sweedler duals. The pairing is called {\it
perfect} if these maps are injections.

\smallbreak
Let $G$ be a connected, simply connected simple complex algebraic group and
$\g= \Lie (G)$ the Lie algebra of $G$. We fix $\h \subseteq \g$ a  
Cartan subalgebra and $\Phi$ the root system associated to $\h$
with simple roots $\Pi=\{\alpha_1,\ldots,\alpha_n\}$, where $n= \rk(\g):=\dim(\h)$.
Let 
 $(-,-)$ be the symmetric bilinear form over $\h^*$ induced by the Killing form. 
Then, the Cartan matrix associated $A=(a_{ij})_{1\leq i,j\leq n} \in \ZZ^{n\times n}$ 
to $\lieg$ is given by $a_{ij} = \frac{2(\alpha_j,\alpha_i)}{(\alpha_i,\alpha_i)}$. 
If we write $d_i=\frac{(\alpha_i,\alpha_i)}{2}$ and $D=\diag(d_1,\ldots,d_n)$, then 
$(\alpha_i,\alpha_j)= d_i a_{ij}$ and $DA$ is
symmetric.
The fundamental weights $\omega_1, \ldots, \omega_n$ are 
given by the property $(\omega_i,\alpha_j)=d_i\delta_{ij}$ for all $1\leq i \leq n$. Then, 
$\alpha_i = \sum_{j=1}^{n} a_{ji}\omega_j $ for all $1\leq i\leq n$. 
We denote by
$P=\sum_{i=1}^{n}{\mathbb{Z}\omega_i}$ the weight lattice, $P_+$ the positive weights, 
$Q=\sum_{i=1}^{n}{\mathbb{Z}\alpha_i}$ the root lattice, $Q_+$ the positive roots and $\W$ the Weyl group associated to $\Phi$.
The bilinear form $(-,-)$ defines a $\mathbb{Z}$-pairing over $P\times Q$.

\smallbreak
Let $q$ be an indeterminate, $R=\QQ[q,q^{-1}]$ and $\QQ(q)$ its field of fractions.
Let $\epsilon$ be an $\ell$th root of unity with $\ord \epsilon = \ell$ odd and $3\nmid \ell$ if $G$ is of type $G_{2}$. 
If $\chi_{\ell}(q)$ denotes the $\ell$th cyclotomic polynomial, 
then $R/[\chi_{\ell}(q)R] = \QQ(\epsilon)$. We denote $q_{i} = q^{d_{i}}$ for all $1\leq i \leq n$.

\smallbreak
For $n > 0$ define
\begin{align*}
(n)_{q} &= \frac{q^{n}-1}{q-1}=q^{n-1} + \cdots +q + 1,\quad
(n)_{q}!  =  (n)_{q}(n-1)_{q}\cdots (2)_{q}(1)_{q}\text{ and }(0)_{q}=1,\\
[n]_{q} & =\frac{q^n-q^{-n}}{q-q^{-1}},\quad 
[n]_{q}!  =[n]_{q}[n-1]_{q}\cdots[1]_{q}\quad\text{ and }\quad[0]_{q}=1,\\
\binom{n}{k}_{q}&= \frac{(n)_{q}}{(k)_{q}(n-k)_{q}},\qquad {\bigg[{n \atop k}\bigg]}_q= \frac{{[n]}_q!}{{[k]}_q! {[n-k]}_q!}.
\end{align*}

\section*{Acknowledgements}
The authors want to thank M. Costantini and F. Gavarini for fruitful conversations, and the referee for 
his/her suggestions to improve the paper.

\section{Twisted quantum groups} \label{sec:twistedqg}
In this section we recall the definition of the twisted (multiparameter simply connected) 
quantum enveloping algebra $\QEphicheck$, its divided power algebra and the twisted quantum function algebra 
$\OO_{q}^{\varphi}(G)$. The former is isomorphic to the multiparameter 
quantum enveloping algebra defined in \cite{R} and \cite{ckp}, see \cite{cv1}, and the latter is 
introduced by Costantini and Varagnolo in \cite{cv2}. We follow mainly \cite{cv2} for the description.

These algebras depend on a $\QQ$-linear map on the weight lattice that induces a deformation
on the coproduct of $\QEcheck$, and on the product of $\OO_{q}(G)$. This deformation is given by a 
multiplicative 2-cocycle on $\OO_{q}(G)$ and {\it resembles} a twist deformation on $\QEcheck$.
For this reason, we call them {\it twisted} quantum function algebras and {\it twisted}
quantum enveloping algebras, respectively.
We describe also the corresponding objects at roots of unity and some basic
properties such as PBW basis, a Hopf algebra pairing and the quantum Frobenius map. In particular,
twisted quantum function algebras at roots of unity fit into an exact sequence of Hopf algebras.

Throughout we omit the supraindex $\varphi$ when $\varphi=0$ on the quantum function
algebras and on the corresponding maps if
no possible confusion arise.

\subsection{The twisting map $\varphi$}
Consider the $\QQ$-linear space $\QQ P=\sum_{i=1}^{n}{\mathbb{Q}\omega_i}$ 
spanned by the weigths and define a $\QQ$-linear map satisfying:
\begin{equation}\label{eq:phi}
\begin{cases}
(\varphi x,y)=-(x, \varphi y), & \forall x,y \in \mathbb{Q}P, \\
\varphi \alpha_i = \delta_i = 2 \tau_i, & \tau_i \in P, i=1,...,n,   \\
\frac{1}{2} (\varphi \lambda, \mu) \in \mathbb{Z}, & \forall \lambda, \mu \in P,
\end{cases}
\end{equation}
where
$(-,-)$ is consider as a linear extension of the $\mathbb{Z}$-pairing over $P\times Q$ to a symmetric bilinear form
$\mathbb{Q}P \times \mathbb{Q}P \rightarrow \mathbb{Q}$. In particular, $\varphi$
is antisymmetric with respect to this form.

According to the first two conditions we have that $(2\tau_i , \alpha_j)=-(2\tau_j, \alpha_i)$. Writing
$\tau_i=\sum_{j=1}^{n}{x_{ji}\omega_j}=\sum_{j=1}^{n}{y_{ji}\alpha_j}$ with $x_{ji}, y_{ij} \in \ZZ$ for all
$1\leq i,j\leq n$, it follows that
$$
d_j x_{ji}= \left(\sum_{k=1}^{n}{x_{ki}\omega_k,\alpha_j}\right) = -\left(\sum_{l=1}^{n}{x_{lj}\omega_l,\alpha_i}\right)
= - \sum_{l=1}^{n}{x_{lj}(\omega_l,\alpha_i)}=
 -d_ix_{ij}.
$$
 If we denote $X=(x_{ij})_{1\leq i,j\leq n}$, $Y=(y_{ij})_{1\leq i,j\leq n} \in \ZZ^{n\times n}$,
then $AY=X$ and $DX=(d_ix_{ij})_{1\leq i,j\leq n}$ is antisymmetric. In particular,
$x_{ii}=0$ for all $1\leq i \leq n$ and 
$\varphi$ depends on at most $\frac{n(n-1)}{2}$ integer parameters. 
Moreover, by \cite[Lemma 2.1]{cv1} the matrix $A + 2X$ is invertible and
the maps $1\pm \varphi : \QQ P \to \QQ P$ are $\QQ$ isomorphisms that 
satisfy that 
$$ ((1+\varphi)^{\pm 1}\lambda, \mu)= (\lambda, (1-\varphi)^{\pm 1}\mu) \qquad \text{ for all }\lambda,\ \mu \in P.
$$
Write $r=(1+\varphi)^{-1}$, $\bar{r} = (1-\varphi)^{-1}$. Note that if 
$\lambda \in r(P)$ and $\mu \in P$, then $(\lambda,\mu)\in \frac{1}{\det(A+2X)} \ZZ$.
Let $u$ be an element contained in the algebraic clausure of $\QQ(q)$ such that $q=u^{\det (A+2X)}$.
If $z \in \frac{1}{\det(A+2X)}\ZZ$, then we write $q^{z}$ for $u^{z\det(A + 2X)}$.

\subsection{Twisted quantum enveloping algebras}\label{subsec:tmqea}
Let $Q\subseteq M \subseteq P$ be a lattice.
For convenience,  we recall the definition of the one-parameter quantum 
enveloping algebras $\QEM $, see \cite[I.6.5]{K-K}. 
\begin{definition}\label{algcuan}
$\QEM$ is the  $\QQ(q)$-algebra generated by the elements 
$\{E_i,F_i\}_{i=1}^{n}$,\  linebreak $\{K_\lambda:\lambda \in M\}$ 
satisfying the relations
\begin{eqnarray}
K_0=1,\qquad   K_\lambda K_\mu & =  &K_{\lambda+\mu} =K_\mu K_\lambda\quad 
\text{ for all }\quad \lambda , \mu \in M, \nonumber \\
K_\lambda E_j K_{-\lambda} & = &q^{(\lambda,\alpha_j)}E_j, \nonumber\\ 
K_\lambda F_j K_{-\lambda} &= & q^{-(\lambda,\alpha_j)}F_j, \nonumber\\ 
\nonumber
E_iF_j-F_jE_i &= &\delta_{ij}\dfrac{K_i-K_i^{-1}}{q_i-q_i^{-1}},\\
\nonumber
\sum\limits_{l=0}^{1-a_{ij}}(-1)^l\begin{bmatrix}
1-a_{ij}\\
l
\end{bmatrix}_{q_i} E_i^{1-a_{ij}-l} E_jE_i^l &= & 0 \text{\quad} (i\neq j),\\
\nonumber
\sum\limits_{l=0}^{1-a_{ij}}(-1)^l\begin{bmatrix}
1-a_{ij}\\
l
\end{bmatrix}_{q_i} F_i^{1-a_{ij}-l} F_jF_i^l &= & 0 \text{\quad} (i\neq j).
\end{eqnarray}
\end{definition}

It is well-known that $\QEM$ is a Hopf algebra with its 
comultiplication defined by setting
$E_i$ to be $(1,K_{\alpha_{i}})$-primitive and $F_i$ to be
$(K_{-\alpha_{i}},1)$-primitive, for all 
$1\leq i \leq n$.
Using the map $\varphi$, one may define a different coproduct, 
counit and antipode on $\QEM$ as follows
(see \cite[\S 1.3]{cv2})
$$\begin{cases}
\Delta_{\varphi}(E_i) = E_i \otimes K_{\tau_i}+K_{\alpha_i-\tau_i} \otimes E_i,\\
\Delta_{\varphi}(F_i) = F_i \otimes K_{-\alpha_i-\tau_i}+K_{\tau_i} \otimes F_i,\\
\Delta_{\varphi}( K_\lambda)=K_\lambda \otimes K_\lambda,
\end{cases} 
\begin{cases}
\eps_\varphi (E_i) = 0,\\
\eps_\varphi (F_i) = 0,\\
\eps_\varphi (K_\alpha) = 1,
\end{cases}
\begin{cases}
\SS_\varphi (E_i) = -K_{-\alpha_i}E_i,\\
\SS_\varphi (F_i) = -F_i K_{\alpha_i},\\
\SS_\varphi (K_\lambda) = K_{-\lambda}.
\end{cases}
$$
Note that the coproduct is well-defined by \eqref{eq:phi}. With this new structure, $\QEM$ is again a Hopf algebra which is   
denoted by $\QEMphi$. Clearly, $\QEMphi= \QEM$ if $\varphi=0$.
We write $\QEPphi =  \QEphicheck$ and $ \QEQphi= \QEphi$.

\begin{obs}\label{rmk:releiki}
From the defining relations we have that $K_{\tau_i}E_i=E_iK_{\tau_i}$ and
$K_{\tau_i}F_i=F_iK_{\tau_i}$ for all 
$1\leq i \leq n$. Indeed, 
$$K_{\tau_i}E_i=\prod_{j=1}^{n}{K_{\omega_j}^{x_{ji}}}E_i=
\left(\prod_{j=1}^{n}{q_j^{x_{ji}(\omega_j,\alpha_i)}}\right)E_i
\left(\prod_{j=1}^{n}{K_{\omega_j}^{x_{ji}}}\right)= 
\left(\prod_{j=1}^{n}{q_j^{d_j x_{ji}\delta_{ij}}}\right)E_iK_{\tau_i}=E_iK_{\tau_i}.$$ 
The second assertion follows analogously. 
\end{obs}

\begin{definition}\label{def:borel}\cite[\S 1.4]{cv2}
Let $\QEPbmasphi$ and $\QEbmasphi$  be the Hopf subalgebras of 
$\QEphicheck$ generated by the elements $K_{\lambda}, E_{i}$
with  $\lambda \in P$ and $\lambda \in Q$, respectively. 
Similarly, let
$\QEPbmenphi$ and $\QEbmenphi$ be the Hopf subalgebras of 
generated by the elements
$K_{\lambda}, F_{i}$, with $\lambda \in P$, and $\lambda \in Q$, respectively.
The algebra $\QEphi$ is the Hopf subalgebra generated by 
$K_{\lambda}, E_{i}, F_{i}$ with $\lambda \in Q$ and $1\leq i\leq n$.
\end{definition}

\subsection{Pairings, Borel subalgebras and integer forms}
By 
\cite[$\S$ 2]{cv1}, see also \cite[\S 1.4]{cv2}, \cite[\S 3]{dl}, there exist perfect Hopf pairings
$\pi_{\varphi}:  \QUhatphi(\mathfrak{b}_{-})^{\cop}\times \QUhatphi(\mathfrak{b}_+)\to
\QQ(u)$ and  $\bar{\pi}_{\varphi}:  \QUhatphi(\mathfrak{b}_+)^{\cop}\times
\QUhatphi(\mathfrak{b}_-)\to \QQ(u)$. These are given by 
$$
\begin{cases}
\pi_\varphi(K_\lambda,K_\mu)=q^{(r(\lambda),\mu)},\\
\pi_\varphi(K_\lambda,E_i)=\pi_\varphi(F_i,K_\lambda)=0,\\
\pi_\varphi(F_i,E_j)=\dfrac{\delta_{ij}}{q_i-q_i^{-1}}q^{(r(\tau_i),\tau_i)},
\end{cases}\qquad 
\begin{cases}
\overline{\pi}_\varphi(K_\lambda,K_\mu)=q^{-(\overline{r}(\lambda),\mu))},\\
\overline{\pi}_\varphi(E_i,K_\lambda)=\overline{\pi}_\varphi(K_\lambda,F_i)=0,\\
\overline{\pi}_\varphi(E_i,F_j)=\dfrac{\delta_{ij}}{q_i^{-1}-q_i}
q^{-(\overline{r}(\tau_i),\tau_i)},
\end{cases}
$$
for $\lambda, \mu \in P$ and $1\leq i,j \leq n$. The pairing $\overline{\pi}_\varphi$ 
can be obtained from $\pi_{-\varphi}$ by the conjugation of the Hopf algebra
$\QQ$-anti-isomorphism $\zeta_\varphi: \QUhatphi(\g) \to \check{U}^{-\varphi}_q (\g)$ 
given by $E_i\mapsto F_i$, $F_i \mapsto E_i$, $K_\lambda \mapsto K_{-\lambda}$ and $q 
\mapsto q^{-1}$. Clearly, $\zeta_{\varphi}$ maps 
$\QUhatphi(\mathfrak{b}_+)$ into $\check{U}^{-\varphi}_q (\mathfrak{b}_-)$ and
$\overline{\pi}_\varphi = \zeta_{\varphi} \circ \pi_{-\varphi} \circ 
(\zeta_{\varphi}\ot \zeta_{\varphi})$. If $\varphi=0$, denote by $\pi_{0}$ the corresponding bilinear form.
Using these pairings we will define four $R$-Hopf algebras that will be needed later. 

\smallbreak
Fix a reduced
expresion of the longest element $w_{0}= s_{i_{1}}\cdots s_{i_{N}}$ in the Weyl group $\W$ and
consider the total ordering on $\Phi_{+}$ given by 
$$
\beta_{1} = \alpha_{i_{1}},\qquad \beta_{2} = \alpha_{i_{1}}\alpha_{i_{2}},\qquad \ldots \qquad
\beta_{N} = \alpha_{i_{1}}\alpha_{i_{2}}\cdots \alpha_{i_{N}}.
$$

The braid group $B_{\W}$ associated to $\W$ acts on $\QEphicheck$ via the Lusztig automorphisms $T_{i_{j}}$ for $1\leq j \leq N$, 
and one may define
the root vectors  
$$ E_{\beta_{k}} = T_{i_{1}}T_{i_{2}}\cdots T_{i_{k-1}} (E_{k})\qquad\text{ and }\qquad
F_{\beta_{k}} = T_{i_{1}}T_{i_{2}}\cdots T_{i_{k-1}} (F_{k}).
$$
For $s\in \NN$, $1\leq i \leq n$, $1\leq k \leq N$ and $G_{i} = E_{i}$ or $F_{i}$ define
$$ G_{i}^{(s)}= \frac{G_{i}^{s}}{[s]_{q_{i}}!},\qquad\text{ and }\qquad
G_{\beta_{k}}^{(s)} = T_{i_{1}}T_{i_{2}}\cdots T_{i_{k-1}} (G_{k}^{(s)}).
$$
For $\alpha\in \Phi_{+}$, let
\begin{align*}
q_{\alpha} &= q^{\frac{(\alpha,\alpha)}{2}},&\tau_{\alpha} &= \frac{1}{2}\varphi(\alpha),
& e_{\alpha}^{\varphi} = (q_{\alpha}^{-1} - q_{\alpha})E_{\alpha}K_{-\tau_{\alpha}},\\
e_{i}^{\varphi} &= e^{\varphi}_{\alpha_{i}},&f_{i}^{\varphi} &= f^{\varphi}_{\alpha_{i}},
& f_{\alpha}^{\varphi} = (q_{\alpha} - q_{\alpha}^{-1})F_{\alpha}K_{-\tau_{\alpha}}.
\end{align*}

\begin{definition}\label{def:Rborel}
Denote by $\Rqphi[B_-]^\prime$ and $\Rqphi[B_-]^{\prime\prime}$ the $R$-subalgebras 
of $\QUhatphi(\mathfrak{b}_+)^{\op}$ and $\QUhatphi(\mathfrak{b}_+)^{\cop}$, respectively, 
generated by the elements $e^\varphi_\alpha$ and $K_{(1-\varphi)\omega_i}$ for 
$1\leq i \leq n$ and $\alpha \in \Phi_{+}$.
Similarly, let $\Rqphi[B_+]^\prime$ and $\Rqphi[B_+]^{\prime\prime}$ be the $R$-subalgebras 
of $\QUhatphi(\mathfrak{b}_-)^{\op}$ and $\QUhatphi(\mathfrak{b}_-)^{\cop}$, 
generated by the elements $f^\varphi_\alpha$ and $K_{(1+\varphi)\omega_i}$ for 
$1\leq i \leq n$ and $\alpha \in \Phi_{+}$.  
\end{definition}

By restriction, we get the following pairings
\begin{align*}
 \pi_\varphi^\prime & : \QUhatphi(\mathfrak{b}_-) \otimes_{R}  \Rqphi[B_-]^\prime	\to \QQ(q), & 
  \pi_\varphi^{\prime\prime} & :  \Rqphi[B_+]^{\prime\prime} \otimes_{R} \QUhatphi(\mathfrak{b}_+)\to \QQ(q),\\
  \bar{\pi}_\varphi^\prime & : \QUhatphi(\mathfrak{b}_+) \otimes_{R}  \Rqphi[B_+]^\prime	\to \QQ(q),&
 \bar{\pi}_\varphi^{\prime\prime} & :  \Rqphi[B_-]^{\prime\prime} \otimes_{R} \QUhatphi(\mathfrak{b}_-)\to \QQ(q). 
\end{align*}

They are given by 
\begin{align}\label{eq:pairphipi}
\pi_\varphi^\prime (K_\lambda,K_{(1-\varphi)\mu})&=q^{(\lambda,\mu)},&
\pi_\varphi^\prime (F_j,e^\varphi_i)&=-\delta_{ij}, \notag \\
\pi_\varphi^{\prime\prime} (K_{(1+\varphi)\mu},K_\lambda)
& =q^{(\mu,\lambda)}, & \pi_\varphi^{\prime\prime} (f^\varphi_i,E_j)& =\delta_{ij}, \notag \\
\overline{\pi}_\varphi^{\prime} (K_\lambda,K_{(1+\varphi)\mu})
& =q^{-(\lambda,\mu)}, &\overline{\pi}_\varphi^\prime (E_j,f^\varphi_i)&=-\delta_{ij},\\
\overline{\pi}_\varphi^{\prime\prime}
(K_{(1-\varphi)\mu},K_\lambda) & =q^{-(\mu,\lambda)}, & \overline{\pi}_\varphi^{\prime\prime} (e^\varphi_i,F_j)& =\delta_{ij}.\notag
\end{align}

By \cite{L2}, one may take as bases of $U_q^\varphi(\mathfrak{b}_+)$ and 
$U_q^\varphi(\mathfrak{b}_-)$, respectively, the elements
\begin{align*}
\xi_{m,t}&=\prod\limits_{j=N}^{1}E_{\beta_j}^{(m_j)}\prod\limits_{i=1}^{n} 
\binom{K_{\alpha_{i}};0}{t_i}K_{\alpha_{i}}^{-\left[\dfrac{t_i}{2}\right]},
&
\eta_{m,t}&=\prod\limits_{j=N}^{1}F_{\beta_j}^{(m_j)}
\prod\limits_{i=1}^{n} \binom{K_{\alpha_{i}};0}{t_i}K_{\alpha_{i}}^{-\left[\dfrac{t_i}{2}\right]},
\end{align*}

where $[ \, ]$ represents the integer part function and $\binom{K_{\alpha_{i}};0}{t_i}=\prod\limits_{s=1}^{t_i}\frac{K_{\alpha_i}^{-s+1}-1}{q^s-1}$.

\subsubsection*{Divided power algebras}\label{subsubsec:dpa} 
We describe now an integer form of $\QEphi$, which is used to define 
the algebra $\QUphie(\g)$ at the root of unity $\epsilon$. 
\begin{definition}\label{def:gamma}
Let $\Gamma^\varphi(\mathfrak{b}_+)$ and $\Gamma^\varphi(\mathfrak{b}_-)$ be
the $R$-submodules of $\QEphi$ given by
\begin{align*}
\Gamma^\varphi(\mathfrak{b}_+) & = 
\{u \in \QEbmasphi\ \vert \ \pi^{\prime\prime}_{\varphi}({\Rqphi[B_+]^{\prime\prime}}^{\cop}\ot u)\subset R\},\\ 
\Gamma^\varphi(\mathfrak{b}_-) & = 
\{u \in \QEbmenphi\ \vert \ \pi^{\prime}_{\varphi}(u\ot {\Rqphi[B_-]^{\prime}}^{\op})\subset R\}.
\end{align*}
\end{definition}
It is known that the sets $\{\xi_{m,t}\}$ and $\{\eta_{m,t}\}$ are $R$-bases of 
$\Gamma^\varphi(\mathfrak{b}_+) $ and $\Gamma^\varphi(\mathfrak{b}_-)$, respectively. This implies that
both are algebras isomorphic to $\Gamma(\mathfrak{b}_+) $ and $\Gamma(\mathfrak{b}_-)$.
Moreover, they are also subcoalgebras with the coproduct given by
\begin{equation}\label{eq:coproGphi}
\begin{cases}
\Delta_\varphi E_i^{(p)}=\sum\limits_{r+s=p}q_i^{-rs}E_i^{(r)}K_{s(\alpha_i-\tau_{i})}\otimes E_i^{(s)}K_{r\tau_i},\\
\Delta_\varphi F_i^{(p)}=\sum\limits_{r+s=p}q_i^{-rs}F_i^{(r)}K_{s\tau_i}\otimes F_i^{(s)}K_{-r(\alpha_i+\tau_i)},\\
\Delta_\varphi\dbinom{K_i;0}{t}=
\sum\limits_{r+s=t}^{}K_i^s q_i^{-rs}\dbinom{K_i;0}{r}\otimes \dbinom{K_i;0}{s}.
\end{cases}
\end{equation}

Again by restriction, we get the Hopf pairings 
\begin{align*}
 \pi_\varphi^\prime & : \Gamma^\varphi(\mathfrak{b}_-)  \otimes_{R}  \Rqphi[B_-]^\prime	\to R, & 
  \pi_\varphi^{\prime\prime} & :  \Rqphi[B_+]^{\prime\prime} \otimes_{R} \Gamma^\varphi(\mathfrak{b}_+) \to R,\\
  \bar{\pi}_\varphi^\prime & : \Gamma^\varphi(\mathfrak{b}_+)  \otimes_{R}  \Rqphi[B_+]^\prime	\to R,&
 \bar{\pi}_\varphi^{\prime\prime} & :  \Rqphi[B_-]^{\prime\prime} \otimes_{R} \Gamma^\varphi(\mathfrak{b}_-) \to R.  
\end{align*}
By \cite[Lemma 1.12]{cv2}, the algebras $ \Rqphi[B_\pm]^\prime$, $ \Rqphi[B_\pm]^{\prime\prime}$ admit a 
Hopf algebra structure such that the pairings above become perfect Hopf algebra pairings.
Moreover, we have that $\Rqphi [B_\pm]^\prime\simeq\Rqphi [B_\pm]^{\prime\prime}$ as Hopf algebras.

\begin{definition}\cite[\S 3.4]{dl}
The algebra $\Gphig$ is the $R$-subalgebra of $\QEphi$ generated by $\Gphibp$ and $\Gphibm$.
In particular, it is generated
by the elements 
\begin{align*}\label{defGphig}
K_{\alpha_i}^{-1} & & (1\leq i \leq n), & & E_i^{(t)}& & (t\geq 1, 1\leq i\leq n),\\
\dbinom{K_{\alpha_i};0}{t}& & (t\geq 1, 1\leq i\leq n), & & F_i^{(t)}& & (t\geq 1, 1\leq i\leq n).
\end{align*}
\end{definition}



\subsection{Twisted quantum function algebras}\label{subsec:tmwfa}
In this subsection, we introduce the dual algebras $\Ophieq$ of $\QEphi$ and
$\Rqphi[G]$ of $\Gphig$.
They are obtained as the submodules generated by the matrix 
coefficients of representations of type one.

Let $\C_{\varphi}$ be the full faithfull subcategory in $\QEphi$-mod consisting of finite-dimensional 
modules on which the elements $K_{\alpha_{i}}$ act diagonally by powers of $q$. Then $\C_{\varphi}$ is a 
tensor category which is strict. Denote by $\Ophieq$ the $\QQ(q)$-submodule of $\Hom_{\QQ(q)}(\QEphi,\QQ(q))$ spanned by all 
the matrix coefficients of objects in $\C_{\varphi}$. Then  $\Ophieq$ is a $\QQ(q)$-Hopf algebra
with the usual structure. Given $V \in \C_{\varphi}$, $v\in V$ and $f\in V^{*}$, then 
the matrix coefficient $c_{f,v}: \QEphi\to \QQ(q)$ is defined by $c_{f,v}(x) =f(x\cdot v)$
for all $x\in \QEphi$. Then 
we have
$$ \com(c_{f,v}) (x\ot y) = c_{f,v}(xy)\qquad\text{ and }\qquad
m_{\varphi}(c_{f,v}\ot c_{g,w})= c_{f\ot g,v\ot w},$$
for $V,W\in \C_{\varphi}$, $v\in V$, $f\in V^{*}$, $w\in W$, $g,\in W^{*}$ and $x,y \in \QEphi$. 

\smallbreak
For $\Lambda \in P_{+}$, let $L(\Lambda)$ be a simple highest weight module of $\QEphi$. 
Then, $ L(\Lambda) = \bigoplus  L(\Lambda)_{\lambda}$ is a graded module and by the Peter-Weyl Theorem 
we have that 
$\Oq =\bigoplus_{\Lambda \in P_{+}} L(\Lambda) \otimes L(\Lambda)^{*}$,
where $ L(\Lambda)^{*} \simeq  L(-\omega_{0}\Lambda)$.
If $v\in L(\Lambda)_{\mu}$, $f\in L(\Lambda)_{-\lambda}$, then
write $\com(c_{f,v}) = \sum_{i} c_{f,v}^{-\lambda,\nu}\ot c_{f,v}^{-\nu,\mu} \in \bigoplus_{\nu} 
\Oq_{-\lambda,\nu}\ot \Oq_{-\nu,\mu}$. Since $\Ophieq$ equals $\Oq$ as coalgebra, we keep this notation for the coproduct on 
$\Ophieq$.

\begin{lema}\label{lem:deforprp}\cite{ls} For $i=1,2$ and 
$\Lambda_{i}\in P_{+}$, $v_{i}\in L(\Lambda_{i})_{\mu_{i}}$, $f_{i}\in L(\Lambda_{i})_{-\lambda_{i}}$ we
have
$$ m_{\varphi}(c_{f_{1},v_{1}}\ot c_{f_{2},v_{2}}) = q^{\frac{1}{2}((\varphi(\mu_{1}),\mu_{2})-
(\varphi(\lambda_{1}),\lambda_{2}))}m(c_{f_{1},v_{1}}\ot c_{f_{2},v_{2}}).
$$\qed
\end{lema}

\begin{obs}\label{rmk:bigrad}
Following \cite[\S 2]{hlt}, the quantum function algebra $\Oq$ is a $P$-bigraded Hopf algebra. 
In particular, if $f\in L(\Lambda)_{\lambda}^\ast$, $v\in  L(\Lambda)_{\mu}$ and $f(v)\neq 0$ 
we have that $f(v)=f(1\cdot v)=f(K^{-1}_{\alpha_i}K_{\alpha_i}\cdot v)
=f(\SS(K_{\alpha_i})K_{\alpha_i}\cdot v)=(K_{\alpha_i}\cdot f)
(K_{\alpha_i}\cdot v)=q_i^{(\lambda,\alpha_i)+(\mu,\alpha_i)}f(v)$
for all $1\leq i \leq n$, 
which implies that $\lambda = -\mu$ if $f(v)\neq 0$.

If we define the anti-symmetric bicharacter
$p:P\times P\to \QQ(q)$ by $p(\lambda_{1},\lambda_{2}) = 
q^{-\frac{1}{2}(\varphi(\lambda_{1}),\lambda_{2})}$, then it induces a group $2$-cocycle $\tilde{p}$ 
on $P\times P$ given
by  
$$\tilde{p}((\lambda_{1},\mu_{1}),(\lambda_{2},\mu_{2})) = p(\lambda_{1},\lambda_{2})p(\mu_{1},\mu_{2})^{-1}=
q^{\frac{1}{2}((\varphi(\mu_{1}),\mu_{2})-(\varphi(\lambda_{1}),\lambda_{2}))},$$
and  by \cite[Theorem 2.1]{hlt}, $\Ophieq$ is isomorphic to the deformed $P$-bigraded Hopf algebra $\Oq_{p}$ where the 
product is given 
$$ m_{\varphi}(c_{f_{1},v_{1}}\ot c_{f_{2},v_{2}}) = p(\lambda_{1},\lambda_{2})
p(\mu_{1},\mu_{2})^{-1}m(c_{f_{1},v_{1}}\ot c_{f_{2},v_{2}}),
$$
for $\Lambda_{i}\in P_{+}$, $v_{i}\in L(\Lambda_{i})_{\mu_{i}}$, $f_{i}\in L(\Lambda_{i})_{-\lambda_{i}}$.
\end{obs}

\begin{cor}\label{cor:cocycle}
$\Ophieq$ is a $2$-cocycle deformation of $\Oq$. The $2$-cocycle $\sigma: \Oq\ot \Oq \to \QQ(q)$ is given by
the formula 
$$\sigma(c_{f_{1},v_{1}}, c_{f_{2},v_{2}}) = \eps(c_{f_{1},v_{1}})\eps(c_{f_{2},v_{2}})
q^{-\frac{1}{2}(\varphi(\lambda_{1}),\lambda_{2})}$$
for 
$\Lambda_{i}\in P_{+},\ v_{i}\in L(\Lambda_{i})_{\mu_{i}}$, $f_{i}\in L(\Lambda_{i})_{-\lambda_{i}}$,
and $i=1,2$.
\end{cor}

\pf Denote $\chi(\lambda_{1},\lambda_{2}) = 
q^{-\frac{1}{2}(\varphi(\lambda_{1}),\lambda_{2})}$. Clearly, $\sigma(x,1)=\sigma(1,x)=\eps(x)$ for all $x\in \Oq$.
We first prove condition (\ref{eq:concocycle}). For $1\leq i\leq 3$, let $c_{f_i,v_i} \in \Oq$
with $f_i\in L(\Lambda_i)_{\lambda_i}$ and $v_i\in L(\Lambda_i)_{\mu_i}$. On one hand, we have
\begin{align*}
&  \sigma((c_{f_2,v_2})_{(1)},(c_{f_3,v_3})_{(1)})\sigma(c_{f_1,v_1},(c_{f_2,v_2})_{(2)}(c_{f_3,v_3})_{(2)}) \\
 & = \sum_{\nu_{1},\nu_{2}} \sigma(c_{f_2,v_2}^{-\lambda_2,\nu_{1}},c_{f_3,v_3}^{-\lambda_3,\nu_{2}})\sigma(c_{f_1,v_1},c_{f_2,v_2}^{-\nu_{1},\mu_2}c_{f_3,v_3}^{-\nu_2,\mu_3})\\
&= \sum_{\nu_{1},\nu_{2}} \eps(c_{f_2,v_2}^{-\lambda_2,\nu_{1}})\eps(c_{f_3,v_3}^{-\lambda_3,\nu_{2}})\chi(\lambda_{2},\lambda_{3})\eps(c_{f_1,v_1})
\eps(c_{f_2,v_2}^{-\nu_{1},\mu_2})\eps(c_{f_3,v_3}^{-\nu_2,\mu_3})\chi(\lambda_{1},\nu_{2}+\nu_{3})\\
&=\eps(c_{f_1,v_1})\eps(c_{f_2,v_2})\eps(c_{f_3,v_3})\chi(\lambda_{2},\lambda_{3})\chi(\lambda_{1},\lambda_{2}+\lambda_{3}).
\end{align*}
On the other hand, 
\begin{align*}
&  \sigma((c_{f_1,v_1})_{(1)}, (c_{f_2,v_2})_{(1)}) \sigma((c_{f_1,v_1})_{(2)})(c_{f_2,v_2})_{(2)}, c_{f_3,v_3}) \\
 & = \sum_{\nu_{1},\nu_{2}} \sigma(c_{f_1,v_1}^{-\lambda_1,\nu_{1}},c_{f_2,v_2}^{-\lambda_2,\nu_{2}})\sigma(c_{f_1,v_1}^{-\nu_1,\mu_1}c_{f_2,v_2}^{-\nu_{2},\mu_2},c_{f_3,v_3})\\
&= \sum_{\nu_{1},\nu_{2}} \eps(c_{f_1,v_1}^{-\lambda_1,\nu_{1}})\eps(c_{f_1,v_1}^{-\lambda_2,\nu_{2}})\chi(\lambda_{1},\lambda_{2})\eps(c_{f_1,v_1}^{-\nu_1,\mu_1})
\eps(c_{f_2,v_2}^{-\nu_{2},\mu_2})\eps(c_{f_3,v_3})\chi(\nu_{1}+\nu_{2},\lambda_{3})\\
&=\eps(c_{f_1,v_1})\eps(c_{f_2,v_2})\eps(c_{f_3,v_3})\chi(\lambda_{1},\lambda_{2})\chi(\lambda_{1} +\lambda_{2},\lambda_{3}).
\end{align*}
Thus, $\sigma$ is a $2$-cocycle on $\Oq$. We prove now that it satisfies the equation given in Lemma \ref{lem:deforprp}. It actually follows by
a direct computation 
using that 
$\chi(0,0)=1$, $\chi(\lambda,0) = \chi(0,\lambda) = 1$, $\sigma^{-1}(c_{f_{1},v_{1}}, c_{f_{2},v_{2}})=
\sigma(\SS(c_{f_{1},v_{1}}), c_{f_{2},v_{2}}) =
\eps(c_{f_{1},v_{1}})\eps(c_{f_{2},v_{2}})
\chi(\mu_{1},\lambda_{2})$ 
for 
 $\Lambda_{i}\in P_{+}$, $v_{i}\in L(\Lambda_{i})_{\mu_{i}}$, $f_{i}\in L(\Lambda_{i})_{-\lambda_{i}}$
and $i=1,2$,  and  $\eps(\Oq_{\lambda,\mu}) = 0 $ if
$-\lambda \neq \mu$:
\begin{align*}
 & m_{\sigma}(c_{f_{1},v_{1}}, c_{f_{2},v_{2}}) =
 \sum_{\nu_{1},\nu_{2},\eta_{1},\eta_{2}} 
 \sigma(c_{f_{1},v_{1}}^{-\lambda_{1},\nu_{1}}, c_{f_{2},v_{2}}^{-\lambda_{2},\nu_{2}})  
 m(c_{f_{1},v_{1}}^{-\nu_{1},\eta_{1}}\ot c_{f_{2},v_{2}}^{\nu_{2},\eta_{2}})
 \sigma^{-1}(c_{f_{1},v_{1}}^{-\eta_{1},\mu_{1}} c_{f_{2},v_{2}}^{-\eta_{2},\mu_{2}})\\
 &\qquad=\sum_{\nu_{1},\nu_{2},\eta_{1},\eta_{2}} \eps(c_{f_{1},v_{1}}^{-\lambda_{1},\nu_{1}})
 \eps(c_{f_{2},v_{2}}^{-\lambda_{2},\nu_{2}})
\chi(\lambda_{1},\lambda_{2}) m(c_{f_{1},v_{1}}^{-\nu_{1},\eta_{1}}\ot c_{f_{2},v_{2}}^{\nu_{2},\eta_{2}})
\eps(c_{f_{1},v_{1}}^{-\eta_{1},\mu_{1}})\eps(c_{f_{2},v_{2}}^{-\eta_{2},\mu_{2}})
\chi(\mu_{1},\eta_{2})\\
 &\qquad= \chi(\lambda_{1},\lambda_{2}) m(c_{f_{1},v_{1}}\ot c_{f_{2},v_{2}})
\chi(\mu_{1},-\mu_{2})\\
 &\qquad= \chi(\lambda_{1},\lambda_{2})\chi(\mu_{1},\mu_{2})^{-1} 
 m(c_{f_{1},v_{1}}\ot c_{f_{2},v_{2}}) = q^{\frac{1}{2}((\varphi(\mu_{1}),\mu_{2})-(\varphi(\lambda_{1}),\lambda_{2}))}m(c_{f_{1},v_{1}}\ot c_{f_{2},v_{2}}) .
\end{align*}
\epf

For more details on twisting, deformation and $r$-matrices, see \cite{hlt}, \cite[\S 2.2]{cv2}.

\begin{definition}\label{def:RphiG}
Let  $\E_\varphi$ be the full faithfull subcategory in 
$\Gphig$-mod whose objects are the free $R$-modules of finite rank such that the elements 
$K_i$ and $\binom{K_i;0}{t}$ 
act by diagonal matrices with eigenvalues  $q_i^m$ and $\binom{m}{t}_{q_i}$ respectively.
Define $\Rqphi [G]$ as the $R$-submodule of $\Hom_{R}(\Gphig, R)$ generated by the matrix coefficients of elements 
in $\E_{\varphi}$. Analogously, we define  $\Rqphi [B_{\pm}]$ 
as the $R$-module generated by the matrix coefficients of elements of the full 
subcategories of $\Gphibp$-mod and $\Gphibm$-mod, respectively.
\end{definition}

Since the categories are strict and tensorial, $\Rqphi [G]$ and $\Rqphi [B_{\pm}]$
are $R$-Hopf algebras. Moreover, by \cite[\S 2.3]{cv2}, we have the isomorphims  
$$
\Rqphi [B_\pm]^\prime\simeq\Rqphi [B_\pm]\simeq\Rqphi [B_\pm]^{\prime\prime}
$$

Consider the linear map 
$\Gphibp\ot_{R} \Gphibm \to \Gphig$ given by the multiplication. The 
dual map composed with the isomorphism above give the injection
\begin{equation}\label{eq:defmu}
\mu_{\varphi}^{\prime\prime} : \Rqphi[G] \to \Rqphi[B_{+}]^{\prime\prime}\ot_{R} \Rqphi[B_{-}]^{\prime\prime}. 
\end{equation}

\begin{lema}{\cite[Lemma 2.5]{cv2}}\label{lem:mu}
The image of 
$\mu_\varphi^{\prime\prime}$ is contained in the $R$-subalgebra 
$\AA^{\prime\prime}_\varphi$ generated by elements the
$1\otimes e_\alpha^\varphi$, $f_\alpha^\varphi\otimes 1$ and $ K_{-(1+\varphi)\lambda}\otimes K_{(1-\varphi)\lambda}$
for $\lambda \in P$, $\alpha \in \Phi_+$.\qed
\end{lema}

Let $\lambda \in P_{+}$ and $v_{\pm \lambda}$ be a highest (resp. lowest) 
weight vector of $L(\lambda)$ (resp. $L(-\lambda)$). 
Let $\phi_{\pm\lambda}$ be the unique element in $L(\pm\lambda)^\ast$, 
such that $\phi_{\pm\lambda}(v_{\pm\lambda})=1$ and vanish over the complement 
$\Gamma(\lieh)$-invariant of $\QQ(q)v_{\pm\lambda}\subset L(\pm \lambda)$. 
Denote by $\psi_{\pm \lambda} = c_{\phi_{\pm \lambda},v_{\pm \lambda}}$ the corresponding
matrix coefficient.

As in \cite{dl}, we define for all $\alpha\in\Phi_+$,  the matrix coefficient  $\psi_{\pm\lambda}^{\pm \alpha}$
by
\begin{align*}
\psi_{\lambda}^{\alpha}(x)&=\phi_{\lambda}((E_\alpha x)\cdot v_{\lambda}),& 
\psi_{-\lambda}^{\alpha}(x)&=\phi_{-\lambda}(( x E_\alpha)\cdot v_{-\lambda}),\\
\psi_{\lambda}^{- \alpha}(x)&=\phi_{\lambda}((x F_\alpha)\cdot  v_{\lambda}),&
\psi_{-\lambda}^{- \alpha}(x)&=\phi_{-\lambda}((F_\alpha x )\cdot  v_{-\lambda}).
 \end{align*}

\begin{rmk}\label{rmk:reci}
$(a)$ Let $\lambda\in P_+$, then  
$\mu_\varphi^{\prime\prime}(\psi_{-\lambda})=K_{-(1+\varphi)\lambda}\otimes K_{(1-\varphi)\lambda}$.

Indeed,
evaluating both expressions in $EM \otimes FN$ where $ EM=\xi_{m_1,0}\eta_{0,t_2}$ and  $FN=\eta_{m_2,0}\xi_{0,t_1}$ for suitable $m_1, t_2, m_2, t_1$
of the basis of $\Gamma^\varphi(\mathfrak{b}_+)$ and $\Gamma^\varphi(\mathfrak{b}_-)$ (c.f. Definition \ref{def:gamma})
respectively, and using \cite[Lemma 4.4 (iv)]{dl} we have that
\begin{equation*}
\langle\mu_\varphi^{\prime\prime}(\psi_{-\lambda}),EM\otimes NF\rangle=\psi_{-\lambda}(EMNF)=\delta_{1,E}\delta_{1,F}MN(-\lambda),
\end{equation*} 
where 
$M(\lambda)=\pi_0(K_{\lambda},M)$ and $N(\lambda)=\bar{\pi}_0(K_{-\lambda},N)$. 
Then $MN(-\lambda)=\pi_0(K_{-\lambda},MN)=\pi_0(K_{-\lambda},M)\pi_0(K_{-\lambda},N)=\pi_0(K_{-\lambda},M)\bar{\pi}_0(K_{\lambda},N)=M(-\lambda)N(-\lambda)$.
Moreover, using \eqref{eq:pairphipi} we have
\begin{equation*}
\langle\mu_\varphi^{\prime\prime}(\psi_\lambda),EM\otimes NF\rangle=\delta_{1,E}\delta_{1,F}M(-\lambda)N(-\lambda)
=\delta_{1,E}\delta_{1,F}\pi^{\prime\prime}_\varphi(K_{-(1+\varphi)\lambda},M)\bar{\pi}^{\prime\prime}_\varphi(K_{(1-\varphi)\lambda},N).
\end{equation*} 
On the other hand, using the pairings $\pi^{\prime\prime}_{\varphi}$ and $\bar{\pi}^{\prime\prime}_{\varphi}$ we have that
\begin{equation*}
\langle K_{-(1+\varphi)\lambda}\otimes K_{(1-\varphi)\lambda},EM\otimes NF\rangle
=\delta_{1,E}\delta_{1,F}\pi^{\prime\prime}_\varphi(K_{-(1+\varphi)\lambda},M)\bar{\pi}^{\prime\prime}_\varphi(K_{(1-\varphi)(\lambda)},N)
\end{equation*}
and the claim follows.

$(b)$ By \cite[Propositions 1.9 \& 2.7]{cv2}, for all $1\leq i\leq n$ we have that 
 \begin{align}\label{eq:reci3teo}
\mu_\varphi^{\prime\prime}(\psi_{-\omega_{i}}^{-\alpha_{i}})  
& =q^{-(\tau_i,\omega_{i})}f^\varphi_{\alpha_{i}} K_{-(1+\varphi)\omega_{i}}\otimes K_{(1-\varphi)\omega_{i}},\\ 
\nonumber \mu_\varphi^{\prime\prime}(\psi_{-\omega_{i}}^{\alpha_{i}})  
& =q^{-(\tau_i,\omega_{i})}K_{-(1+\varphi)\omega_{i}}\otimes K_{(1-\varphi)\omega_{i}}e^\varphi_{\alpha_{i}} .
\end{align}

We check the first formula, the second follows similarly. 
Since $\mu_\varphi^{\prime\prime}(\psi_{-\omega_{i}}^{-\alpha_{i}})=\mu_0^{\prime\prime}(\psi_{-\omega_{i}}^{-\alpha_{i}})$, and by
\cite[Lemma 4.5 (vi)]{dl}, it holds that $\mu_0^{\prime\prime}(\psi_{-\omega_{i}}^{-\alpha_{i}})  
=f_{\alpha_{i}} K_{-\omega_{i}}\otimes K_{\omega_{i}}$, we have  
\begin{align*}
\langle\mu_\varphi^{\prime\prime}(\psi^{-\alpha_i}_{-\omega_i}),EM\otimes NF\rangle & = 
\langle f_{\alpha_i} K_{-\omega_i}\otimes  K_{\omega_i},EM\otimes NF\rangle =
\pi_0^{\prime\prime}(f_{\alpha_i} K_{-\omega_i},EM)\bar{\pi}^{\prime\prime}_0(K_{\omega_i},NF)\\
&= 
\pi_0^{\prime\prime}(f_{\alpha_i} K_{-\omega_i},EM)\bar{\pi}^{\prime\prime}_0(K_{\omega_i},N)\bar{\pi}^{\prime\prime}_0(K_{\omega_i},F)\\
&=\pi_0^{\prime\prime}(f_{\alpha_i} K_{-\omega_i},EM)N(-\omega_i)\delta_{1,F}.
\end{align*}
On the other hand, since 
$\pi^{\prime\prime}_\varphi(f_{\alpha_i}^\varphi K_{-(1+\varphi)\omega_i},EM)=q^{(\tau_i,\omega_i)}\pi^{\prime\prime}_0(f_{\alpha_i}^0 K_{-\omega_i},EM)$
by \cite[Proposition 1.9]{cv2} and \cite[(3.3)]{dl}, using the definitions in \eqref{eq:pairphipi}, 
we obtain 
\begin{align*}
 \langle f_{\alpha_{i}}^{\varphi}K_{-(1+\varphi)\omega_{i}}\ot K_{(1-\varphi)\omega_{i}},EM\otimes NF\rangle & =
 \pi^{\prime\prime}_{\varphi}(f_{\alpha_{i}}^{\varphi}K_{-(1+\varphi)\omega_{i}},EM) \pi^{\prime\prime}_{\varphi}(K_{(1-\varphi)\omega_{i}},NF) \\
 & = \pi^{\prime\prime}_{\varphi}(f_{\alpha_{i}}^{\varphi}K_{-(1+\varphi)\omega_{i}},EM)N(-\omega_i)\delta_{1,F}\\
 & = q^{(\tau_i,\omega_i)}\pi^{\prime\prime}_0(f_{\alpha_i}^0 K_{-\omega_i},EM)N(-\omega_i)\delta_{1,F},
 \end{align*}
and the assertion is proved.
 \end{rmk}

The following lemma is a twisted version of \cite[Lemma 4.1]{dl}.

\begin{lema}\label{lem:pairRG}
$\Rqphi[G]$ coincides with the $R$-Hopf subalgebra of $\QEphi^\circ$ given by the set 
of all linear functions 
$f:\Gphig\rightarrow R$ such that there exists a cofinite ideal  $I\subset \Gphig$ and $N\in \NN$
which satisfy that $f(I)=0$ and $\prod_{p=-N}^{N}(K_i-q_i^p)\in I$ for all $1\leq i \leq n$.
Further, the induced Hopf pairing $\rho$ between $R_q^\varphi[G]$ and $\Gphig$ is non-degenerate.
\end{lema}

\begin{proof} Since $\Gphig = \Ga(\g)$ as algebras, $\Rqphi[G]$ coincides with the set above
by \cite[Lemma 4.1]{dl}. The Hopf algebra structure is the one induced from $\Gphig^\circ$.
The last claim follows from the fact that $\Gphig$ has a PBW-basis and its dual basis lie 
in $\Rqphi[G]$.
\end{proof}

\subsection{Specializations at roots of one} In this subsection we recall the definition 
at roots of unity of the twisted quantum algebras, and state some results that will be needed later.
For all $Q\leq M \leq P$, we define  
$$\QUphie(\g;M)
= U_q^\varphi(\g;M)\ot_{R} \QQ(\epsilon),\ \Gphige:=\Gphig\otimes_R \QQ(\epsilon),\
\OphieeQ:=\Rqphi[G]\otimes_R\QQ(\epsilon).
$$
Note that $\Gphige\simeq \Gphig/[\chi_l(q)\Gphig]$  
and $\OphieeQ\simeq \Rqphi[G]/[\chi_\ell(q)\Rqphi[G]]$, where\linebreak  
$R/[\chi_\ell(q)R]\simeq \QQ(\epsilon)$.  
We denote 
$\QUphie(\g;P):=\QUhatphie(\g)$ and $\QUphie(\g;Q):=\QUphie(\g)$.
For $r\in R$, denote by $\bar{r}$ the image of the canonical projection $R\twoheadrightarrow \QQ(\epsilon)$.

\begin{lema}\label{lem:def2cocicloOe}
$\OphieeQ$ is a 2-cocycle deformation of $\Oc_\epsilon(G)_{\QQ(\epsilon)}$.
\end{lema}

\pf
Let 
$\sigma: \Oq \ot \Oq \to \QQ(q)$ denote the $2$-cocycle defined in 
Corollary \ref{cor:cocycle}. Then, the map $\bar{\sigma}: \OphieeQ\otimes \OphieeQ \to \QQ(\epsilon)$
given by 
$$ \bar{\sigma}(\bar{x},\bar{y})=\overline{{\sigma(x,y)}}\qquad\text{ for all }x,y\in \Oq,$$
is a well-defined $2$-cocycle for $\OphieeQ$, 
where $\bar{x}$ denotes the image of $x\in \Oq$ under the canonical projection $\Ophieq \twoheadrightarrow \OphieeQ$. 
\epf

\begin{obs}\label{rmk:relmodepsgam}
The relations $E_i^\ell=0$, $F_i^\ell=0$, $K_{\alpha_{i}}^\ell=1$ hold in $\Gphige$ for all $1\leq i\leq n$.
Indeed, we have that 
$\prod\limits_{s=1}^{\ell}\left(K_{\alpha_{i}}q^{(-s+1)}-1\right)
=\prod\limits_{s=1}^{\ell}(q^s-1)\dbinom{K_{\alpha_{i}};0}{\ell}$ in $\Gphig$.
If we specialize $q$ at $\epsilon$, then we have $\prod\limits_{s=1}^{\ell}\left(K_{\alpha_{i}}\epsilon^{(-s+1)}-1\right)=0$. 
Since $K_{\alpha_{i}}^\ell-1=\prod\limits_{s=0}^{\ell-1}(K_{\alpha_{i}}-\epsilon^s) = 
\epsilon^{\frac{(\ell-1)\ell}{2}}\prod\limits_{s=0}^{\ell-1}(K_{\alpha_{i}}\epsilon^{-s+1}-1) $, 
we have that $K_{\alpha_{i}}^{\ell}=1 $ as desired. 
The other two relations follow from the fact $(\ell)_\epsilon=0$.
\end{obs}


The following lemma is analogue to \cite[Lemma 6.1]{dl}. 

\begin{lema}\label{lem:pairind} There
exists a perfect Hopf pairing $\bar{\rho}:\OphieeQ \otimes_{\QQ(\epsilon)}\Gphige\to \QQ(\epsilon)$. 
\end{lema}

\begin{proof}
Let $\rho: R_q^\varphi[G]\otimes _{R}\Gphig\to R$ denote the 
pairing defined in Lemma \ref{lem:pairRG}. Then, we may define the pairing 
$\bar{\rho}: \OphieeQ \otimes_{\QQ(\epsilon)}\Gphige\to \QQ(\epsilon)$ via
$\bar{\rho}(\bar{x},\bar{u}) =\overline{(\rho(x,u))}$ for all $x\in R_{q}^{\varphi}[G]$ and 
$u \in \Gphig$, where $\bar{x}$ and $\bar{u}$ denote the images of $x$ and $u$
under the canonical projections
$\Rqphi[G]\to\OphieeQ$ and $\Gphig\to\Gphige$, respectively. 
A direct computation shows that $\bar{\rho}$ is a well-defined map and it is a non-degenerate Hopf pairing. 
\end{proof}

Now we introduce the twisted quantum Frobenius map. For details, see \cite[\S 3]{cv2}.
For $1\leq i\leq n$, let $e_i$, $f_i$ and $h_i$ denote the Chevalley generators of $\g$
and write $e_i^{(m)}:=e_i/m!$, $f_i^{(m)}:=f_i/m!$, 
$\dbinom{h_i}{m}:=\frac{h_i(h_i-1)\cdots (h_i-m+1)}{m!}$  for all $m\geq 0$. 

\begin{lema}\cite[\S 3.2 (i)]{cv2}\label{prop:homofropro}
There is a unique Hopf algebra epimorphism $\Fr:\Gphige\longrightarrow U(\g)_{\QQ(\epsilon)}$ given 
for all $1\leq i\leq n$ and $m>0$, by
\begin{align*}
\Fr (E_i^{(m)})&=e_i^{(m/\ell)}, & \Fr (F_i^{(m)})&=f_i^{(m/\ell)}, & 
\Fr\dbinom{K_{\alpha_{i}};0}{m}&=\dbinom{h_i}{m/\ell}, & \Fr (K_{\alpha_{i}})&=1,
\end{align*}
if $\ell\mid m$ or $0$ otherwise.
Its kernel is the ideal generated by the elements $K_{\alpha_{i}}-1$, $E_i$ and $F_i$.
In particular, there is a Hopf algebra 
monomorphism $^{t}\Fr:\Oc(G)_{\QQ(\epsilon)}\rightarrow\Gphige^\circ$. 
\qed
\end{lema}

Let $\Bbbk$ be a field extension of $\QQ(\epsilon)$. 
We call $\OphieeQ\otimes_{\QQ(\epsilon)}\Bbbk$ the $\Bbbk$-form of $\OphieeQ$.
When $\Bbbk=\CC$ we simply write $\Oephi$.

\begin{lema}\label{lem:subalgcen} \cite[\S 3.3]{cv2}
$\OphieeQ$ contains a central Hopf subalgebra $F_0$ 
isomorphic to $\Oc(G)_{\QQ(\epsilon)}$. 
Moreover, an element of $\OphieeQ$ belongs to   $F_0$ 
if only if it vanishes on $I$ and 
$$F_0=\QQ(\epsilon)\langle \overline{c}_{f,v} \in 
\Ophieq_{\QQ(\epsilon)} \vert f \in \overline{L}(\ell\Lambda)^*_{-\ell v}, v\in \overline{L}(\ell\Lambda)_{\ell\mu}; 
v,\mu \in P_+\rangle,$$  
where $\overline{L}(e\Lambda)$ is the 
$\Gphig$-module $\Gphig v_{e\Lambda}$ with $v_{e\Lambda}$ the highest weight vector of $L(e\Lambda)$.
\qed
\end{lema}

\begin{prop}\label{prop:femodpro}
$\OphieeQ$ is a free $\OO(G)_{\QQ(\epsilon)}$-module of rank $\ell^{\dim \lieg}$.
\end{prop}

\begin{proof} Follows from \cite[Proposition 3.5]{cv2}, \cite{dl} and  \cite{BGS}.
\end{proof}

Let $\overline{\Oephi}$ be the quotient $\Oephi / [\OO(G)^+\Oephi]$ and 
$\pi:\Oephi \longrightarrow \overline{\Oephi)}$ the canonical projection. By 
Proposition \ref{prop:femodpro}, $\overline{\Oephi}$ is a Hopf algebra of 
dimension $\ell^{\dim \lieg}$. Moreover, since 
$\Oephi$ is a free $\OO(G)$-module, it is faithfully flat. Then, by 
\cite[Proposition 3.4.3]{Mo} $\Oephi$ fits into the short exact sequence 
of Hopf algebras. 
$$
1 \longrightarrow \OO(G) \longrightarrow \Oephi \longrightarrow \overline{\Oephi} \longrightarrow 1.
$$

\section{Twisted Frobenius-Lusztig kernels}\label{sec:mflk}
In this section we define and study the twisted Frobenius-Lusztig kernels and the quotients of their duals. They are 
finite-dimensional pointed Hopf algebras which are twist deformations of the usual kernels.

Let $Z^\varphi_0$ be the smaller $B_{\W}$-invariant subalgebra of $\QUphie(\g)$ that contains the elements 
$K_{\ell\alpha}=K_\alpha^\ell$, $E_i^\ell$, $F_i^\ell$ for all $\alpha \in Q$ and $1\leq i \leq n$. 

\begin{theorem}\label{thm:Z0}
\begin{enumerate}
\item[$(i)$] $Z^\varphi_0$ is a central Hopf subalgebra of $\QUphie(\g)$.
\item[$(ii)$] $Z^\varphi_0$ is a polynomial ring  in $\dim \g$ generators, with $n$ generators inverted.
\item[$(iii)$] $\QUphie(\g)$ is a free $Z^\varphi_0$-module of rank $\ell^{\dim \g}$.
\end{enumerate}
\end{theorem}

\begin{proof} The proof follows the same lines as 
\cite[Theorem III.6.2]{K-K}, using that the algebra $W$ spanned by the
elements $K_{\ell\alpha}=K_\alpha^\ell$, $E_i^\ell$, $F_i^\ell$ for all $\alpha \in Q$ and $1\leq i \leq n$ is a Hopf 
subalgebra, and this follows from a simple computation using the $q$-binomial formula. For example, 
$\Delta_{\varphi}(E_i^\ell) = (E_i \otimes K_{\tau_i}+K_{\alpha_i-\tau_i} \otimes E_i)^\ell = 
E_i^\ell \otimes K_{\ell\tau_i}+K_{\ell(\alpha_i-\tau_i)} \otimes E_i^\ell$, 
since $(E_i \otimes K_{\tau_i})(K_{\alpha_i-\tau_i} \otimes E_i)=
\epsilon^{-2d_i}(K_{\alpha_i-\tau_i} \otimes E_i)(E_i \otimes K_{\tau_i})$.
\end{proof}

\begin{definition}\label{def:flkernels}
The \textit{twisted Frobenius-Luzstig kernel} is defined as the quotient 
$$\qephi =\QUphie(\g)/[(Z_0^\varphi)^+\QUphie(\g)].$$ 
\end{definition}

By the theorem above, $\qephi$ is a finite-dimensional pointed Hopf algebra of dimension $\ell^{\dim \g}$ and
$G(\qephi)= \langle K_{\alpha_{i}}|\ 1\leq i\leq n\rangle \simeq (\ZZ/\ell\ZZ)^{n}$.
We denote  $G(\qephi)= \Tt^{\varphi}$.

\begin{lema}\label{lem:reslusyflkiso}
Let $\QUhatphieR(\g)$ be the Hopf subalgebra of $\Gphige$ generated by the 
elements $E_i,F_i$, $K_{\alpha_{i}}$ with $1 \leq i \leq  n$. Then
$\QUhatphieR(\g)$ and $\qephi$ are isomorphic as Hopf algebras.
\end{lema}

\begin{proof} 
By definition, there exists a Hopf epimorphism $\QUhatphieR(\g)\twoheadrightarrow \qephi$
given by $E_i\mapsto E_i$, $F_i \mapsto F_i$ and $K_{\alpha_{i}}\mapsto K_{\alpha_{i}}$
for all $1\leq i\leq n$. Since by Remark \ref{rmk:relmodepsgam}, $\dim \QUhatphieR(\g)\leq  \ell^{\dim \g}$, the
claim follows.
\end{proof}

Adapting the proof of \cite[Theorem III.7.10]{K-K}, we have the following. 
\begin{theorem}\label{thm:oephiisodualnfl}
The Hopf algebras $\overline{\Oephi}$ and $\qephi^{*}$ are isomorphic.
\end{theorem}

\begin{proof}
The pairing defined in Lemma \ref{lem:pairind} induces a perfect Hopf pairing 
$\overline{\Oephi}\otimes_{\QQ(\epsilon)}\QUhatphieR (\g)\to \QQ(\epsilon)$. 
In particular, we have a
Hopf algebra monomorphism $\overline{\Oephi}\hookrightarrow \QUhatphieR(\g)^\ast$. 
Since both algebras have the same dimension, the assertion follows by Lemma \ref{lem:reslusyflkiso}.
\end{proof}

As a consequence of the theorem above, the following sequence of Hopf algebras is exact
$$\xymatrix{1 \ar[r]^{} & \Oc(G)_{\QQ(\epsilon)} \ar[r]^{\iota} &
\OphieeQ \ar[r]^{\pi}
& \qephi^{*} \ar[r]^{} & 1.}
$$

\begin{prop}\label{prop:froblustkertwist}
 $\qephi\simeq \qe^{J}$ for a twist $J \in \QQ(\epsilon)[\Tt^{\varphi}\times \Tt^{\varphi}]$. 
\end{prop}

\begin{proof}
By Lemma \ref{lem:def2cocicloOe}, $\Oephi$ is a $2$-cocycle deformation of $\Oe$. Denote 
this cocycle by $\bar{\sigma}$. Then, if we denote $\mathcal{I}=\Oc(G)^+\Oephi$, it holds that $\bar{\sigma}|_{\mathcal{I}\otimes \Oephi+ \Oephi\otimes \mathcal{I}} = 0$ and 
by Remark \ref{obs:cocycleproy}, we have that $\qephi^{*}$ is a $2$-cocycle deformation of
$\qe^{*}$, where the cocycle is given by the formula $\hat{\sigma}(\pi(x),\pi(y)) = \bar{\sigma}(x,y)$ for all $x,y \in \Oe$. 
We may consider 
$\hat{\sigma}: \qe^{*}\ot \qe^{*}\to \QQ(\epsilon)$ as an element in $\qe\ot \qe$, say $J= \sum_{i} u_{i}\ot u^{i} $.
Then, 
\begin{align*}
\hat{\sigma}(\pi(c_{f_{1},v_{1}})\ot \pi(c_{f_{2},v_{2}})) &= 
\langle J,\pi(c_{f_{1},v_{1}})\ot \pi(c_{f_{2},v_{2}})\rangle = 
\sum_{i}  f_{1}(u_{i}\cdot v_{1}) f_{2}(u^{i}\cdot v_{2}) \\
& = \eps(c_{f_{1},v_{1}})\eps(c_{f_{2},v_{2}})\epsilon^{\frac{1}{2}(\varphi(\lambda_{1}), \lambda_{2})}
= f_{1}(v_{1})f_{2}(v_{2})\epsilon^{\frac{1}{2}(\varphi(\lambda_{1}), \lambda_{2})}, 
\end{align*}
for all
$\Lambda_{i}\in P_{+},\ v_{i}\in L(\Lambda_{i})_{\mu_{i}}$, $f_{i}\in L(\Lambda_{i})_{-\lambda_{i}}$,
and $i=1,2$, where $\langle \cdot, \cdot \rangle$ is the 
perfect pairing given by the evaluation. Thus, the components of $J$ must act diagonally and 
consequently, $J \in \QQ(\epsilon)[\Tt^{\varphi}\times \Tt^{\varphi}]$. 
\end{proof}

\subsection{Subalgebras of $\qephi$}
In this subsection we discuss a parametrization of the Hopf subalgebras of $\qephi$. 
Since $\qephi$ is a pointed Hopf algebra, any Hopf subalgebra is also pointed, and in this
case, it is generated by a subgroup of the group of group-like elements and a subset
of skew-primitive elements.

\begin{lema}\label{lema:subalgparadeuphi}
The Hopf subalgebras of $\qephi$ are parametrized by triples $(I_+,I_-,\Sigma^\varphi)$ where $I_\pm \subset \pm\Pi$
and $\Sigma
^\varphi$ is a subgroup of $G(\qephi)$ such that  $K_{(1\mp\varphi)(\alpha_i)} \in \Sigma^\varphi$
if $\alpha_i\in I_{\pm}$.
Denote $\widetilde{E}_i:=E_iK_{-\tau_i}$ and $\widetilde{F}_j:=K_{(\alpha_j+\tau_j)}F_j$. Then the Hopf subalgebra of
$\qephi$ corresponding to the triple $(I_+,I_-,\Sigma^\varphi)$ is the subalgebra generated by the set
$\{g, \widetilde{E}_i, \widetilde{F}_j|\  g\in \Sigma^{\varphi}, \alpha_{i}\in I_+ \text{ and } \alpha_{j}\in I_- \}$.
\end{lema}      

\begin{proof} The proof follows from 
\cite[Corollary 1.12]{AG}, since $\qephi$ is generated by group-like and skew-primitive elements. In particular, 
$\Delta_\varphi(\widetilde{E}_i)=\widetilde{E}_i\otimes 1+K_{(1-\varphi)(\alpha_i)}\otimes \widetilde{E}_i$,
$\Delta_\varphi(\widetilde{F}_j)=\widetilde{F}_j\otimes 1+ K_{(1+\varphi)(\alpha_j)}\otimes \widetilde{F}_j$.
\end{proof}

Each pair $(I_+,I_-)$ determines a regular Lie subalgebra $\liep$ of $\g$ containing the fixed Cartan subalgebra $\lieh$.
Next we define the corresponding twisted quantum algebras.
\begin{definition}\label{def:gamaparabo}
For every pair $(I_+,I_-)$ with $I_\pm \subset \pm\Pi$, we define
$\Gamma^\varphi(\mathfrak{p})$ as
the subalgebra of $\Gamma^\varphi(\g)$ generated by the elements

\begin{align*}
K^{-1}_{\alpha_i}& &(1\leq i \leq n),\\
\dbinom{K_{\alpha_i};0}{m}&:=\prod_{s=1}^{m}
\left(\dfrac{K_{\alpha_i}q_i^{-s+1}-1}{q_i^s-1}\right)& (m \geq 1, 1 \leq i \leq n),\\
E^{(m)}_j&:=\dfrac{E^m_j}{[m]_{q_j}!}&(m \geq 1, \alpha_j\in I_+),\\
F^{(m)}_k&:=\dfrac{F^m_k}{[m]_{q_k}!}&(m \geq 1, \alpha_k\in I_-).
\end{align*}
\end{definition}         

\begin{prop}\label{prop:gepp} \cite[Proposition 2.3 (a)]{AG}
Let $\Gamma_\epsilon^\varphi(\mathfrak{p}):= 
\Gamma^\varphi(\mathfrak{p})/[\chi_{\ell}(q)\Gamma^\varphi(\mathfrak{p})]
\simeq \Gamma^\varphi(\mathfrak{p}) \otimes_R R/[\chi_\ell(q)R]$ 
denote the $\QQ(\epsilon)$-algebra 
given by the specialization. Then $\Gamma^\varphi_\epsilon(\mathfrak{p})$ is a Hopf subalgebra of $\Gphige$. \qed
\end{prop}

Next we define a family of regular twisted Frobenius-Lusztig kernels. 
\begin{definition}\label{def:nupaFL}
For every pair $(I_+,I_-)$ with $I_\pm \subset \pm\Pi$, we define the
{\it twisted regular Frobenius-Lusztig kerner} $\qephip$ as the subalgebra of $\Gephip$ generated by 
the elements 
$
\{K_{\alpha_i}, E_j, F_k: \, 1 \leq i \leq n, \alpha_j \in I_+, \alpha_k \in I_-\}
$.
\end{definition}

In the following propositions we collect some properties.
\begin{prop}
$\qephip$ is the Hopf subalgebra of $\qephi$ given by $\Gephip \cap\qephi=\qephip$. It corresponds to 
the triple $(I_{+},I_{-},\Tt^{\varphi})$.
\end{prop}

\pf
Follows from  Lemmata \ref{lem:reslusyflkiso} 
and  \ref{lema:subalgparadeuphi}.
\epf

\begin{prop}\label{prop:qephisse}
\begin{itemize}
\item[$(i)$] Let $U(\p)_{\QQ(\epsilon)}:=\Fr (\Gephip)$ and denote $\Fr_{res} = \Fr|_{\Gephip}$. 
Then the following diagram is commutative and all rows are exact sequences of Hopf algebras
\begin{equation}\label{eq:diapara}
\xymatrix{1 \ar[r]^{} & \qephi\ar[r] &
\Gphige\ar[r]^{\Fr}
& U(\g)_{\QQ(\epsilon)} \ar[r] & 1\\
1 \ar[r]^{} & \qephip\ar[r]\ar@{^{(}->}[u] &
\Gephip \ar[r]^{\Fr_{res}}\ar@{^{(}->}[u]
& U(\p)_{\QQ(\epsilon)} \ar[r] \ar@{^{(}->}[u]& 1.}
\end{equation}
\item[$(ii)$]There is a surjective algebra map $\theta:  \Gephip \to \qephip$ such that $\theta|_{\qephip}= \id$.
\end{itemize}
\end{prop}

\pf $(i)$ It follows from \cite{cv2,dl} and Proposition \ref{prop:homofropro} that $\Ker \Fr=\qephi^+\Gphige$.
The proof that $\Gphige^{\co \Fr}=\qephi$ follows from \cite[Lemma 3.4.2]{A1} but using the formula 
\eqref{eq:coproGphi} instead of the formula (1.1.3) in \cite{A1}. 
So the first row is exact.
To prove that the second row is exact, note that 
$\qephip=\qephi\cap\Gephip=\Gphige^{\co \Fr}\cap\Gephip=\Gephip^{\co \Fr_{res}}$ and 
$\Ker \Fr_{res}=\Ker \Fr\cap\Gephip=\qephi^{+}\Gphig\cap\Gephip=\qephip^+\Gephip$. 

$(ii)$ Follows from \cite[Lemma 1.10 \& Proposition 2.6]{AG}. 
\epf

\begin{obs}\label{rmk:liep}
Let $\p$ be the set of primitive elements in $U(\p)_{\QQ(\epsilon)}$. Then, $\p$ is a regular  
Lie subalgebra of $\lieg$, and $\qephip$ is the Frobenius-Lusztig kernel associated to it.
\end{obs}

\begin{prop}\label{prop:uptup}
 $\qephip$ is a twist deformation of $\mathbf{u}_\epsilon(\p)$.
\end{prop}

\begin{proof}
We know that $\qephi \simeq \qe^{J}$ for a twist $J \in \QQ(\epsilon)[\Tt^{\varphi}\times \Tt^{\varphi}]$. Thus,
$J\in \qep\ot \qep$ and  $\qep^{\textit{J}} $ is the subalgebra of $\qe^{J}$
that is isomorphic to the Hopf subalgebra of $\qephi$ which corresponds to the
triple $(I_{+},I_{-},\Tt^{\varphi})$. Hence, $\qep^{J} \simeq \qephip$.
\end{proof}

\subsection{Quotients of $\qephi^{*}$}\label{subsec:Npara}
Denote the $\CC$-form of the twisted
Frobenius-Lusztig kernel just by $\qephi$.
Let $H$ be a Hopf algebra quotient of $\qephi^{*}$. Then, $H^{*}$ is a Hopf subalgebra of
$\qephi$ and whence, by Lemma \ref{lema:subalgparadeuphi}, it is determined by a triple $(I_+,I_-,\Sigma^\varphi)$.
Let $\qephip$ be the regular Frobenius-Lusztig kernel associated to the pair $(I_{+},I_{-})$. Then
$H^{*}\hookrightarrow \qephip \hookrightarrow \qephi$ as Hopf algebras, and consequently we have
a sequence of Hopf algebra epimorphisms 
$$\xymatrix{\qephi^{*} \ar@{->>}[r]^{} & \qephip^{*} \ar@{->>}[r]^{\nu} &  H.}$$

Let $I = I_+\cup I_-$, $I'=I_+\cap I_-$ and $I^{c}= (I_+\cup I_-)^c= I_{+}^c\cap I_{-}^c$.
We define the abelian subgroups $\Tt^\varphi_{I}$ and $\Tt^\varphi_{I'}$ of $\Sigma^{\varphi}$
as follows:
\begin{align*}
\Tt^\varphi_{I} & =  \langle \overline{K}_i:= K_{(1-\varphi)(\alpha_i)}, \widetilde{K}_j:=K_{(1+\varphi)(\alpha_j)}:\
\text{ if }\alpha_i\in I_+, \alpha_j\in I_- \rangle,\\
\Tt^\varphi_{I'} & = \langle K_{\alpha_{i}}:\ \text{ if }\alpha_{i}\in I_+\cap I_-\rangle .
\end{align*}

Note that if $\alpha_{i}\in I_+\cap I_-$, then $K_{\alpha_{i}} \in \Tt^{\varphi}_{I}$. Hence, 
 $\Tt^\varphi_{I'}\subseteq \Tt^\varphi_{I}\subseteq \Sigma^\varphi \subseteq \Tt^{\varphi}$.
Denote $\Tt^\varphi_{I^c}= \Tt^\varphi/\Tt^\varphi_{I}$ and 
$\Omega^\varphi=\Sigma^\varphi / \Tt^\varphi_{I}$; so $\Omega^\varphi \subseteq \Tt^\varphi_{I^c}$.
\begin{definition}\label{def:D}
For all $i\in\{1,\ldots, n\}$ such that $\alpha_i \in (I_+ \cap I_-)^c $,
we define the algebra homomorphism $D_i:\ \qephip\to \CC$ by
\begin{align*}
D_i(E_j)=0=D_i(F_k),\qquad D_i(K_{\alpha_{t}})=\epsilon^{\delta_{it}}\quad
\text{for all }\alpha_j\in I_+,\ \alpha_k\in I_-,\ t\in \{1\ldots n\}.
\end{align*}
\end{definition}

\begin{obs}\label{rmk:notathenDi}
$(a)$ For $1\leq i \leq n$, let $\hat{D}_i \in  \widehat{\Tt^{\varphi}}$ given
by $\hat{D}_i(K_{\alpha_{t}})=\epsilon^{\delta_{it}}$ for all $1\leq t\leq n$. 
Then $\langle \hat{D}_i:\ 1\leq i \leq n\rangle = \widehat{\Tt^{\varphi}}$ and we may 
identify $ (\ZZ/\ell\ZZ)^n\simeq  \widehat{\Tt^{\varphi}}$ by 
$z\mapsto \hat{D}^{z}= \hat{D}_{1}^{z_1}\cdots \hat{D}_{n}^{z_n}$. In particular, 
one has that $\hat{D}_{i} = {D_{i}}|_{\Tt^{\varphi}}$ for all $i\in (I_+ \cap I_-)^c$.

$(b)$ Assume $(I_+\cap I_-)^c = \{\alpha_{i_{1}},\ldots, \alpha_{i_{m}}\}$. For all 
$z\in (\ZZ/\ell\ZZ)^m$, denote
$$D^{z}= D_{i_{1}}^{z_{1}}\cdots D_{i_{m}}^{z_{m}} \in G(\qephip^{*}).$$
If $f\in G(\qephip^{\ast})$ then 
$f=D^{z}$ for some $z\in (\ZZ/\ell\ZZ)^m$. In particular, we may identify 
$$ G(\qephip^{*}) \simeq \widehat{\Tt^{\varphi} / \Tt^{\varphi}_{I'}}  \simeq (\ZZ/\ell\ZZ)^m.$$

$(c)$ Since $\Tt^\varphi_{I'}\subseteq \Tt^\varphi_{I}$, there is a 
group monomorphism $\widehat{\Tt^{\varphi}_{I^{c}}}\hookrightarrow 
\widehat{\Tt^{\varphi} / \Tt^{\varphi}_{I'}}=G(\qephip^{*})$ given for any 
$f\in \widehat{\Tt^{\varphi}_{I^{c}}}$ by the composition
$
\xymatrix{\Tt^{\varphi} / \Tt^{\varphi}_{I'} \ar@{->>}[r] & \Tt^{\varphi}_{I^{c}} \ar[r]^{f}& \CC}.$

$(d)$ 
The inclusions
$\Tt_{I}^{\varphi}\overset{\iota}\hookrightarrow\Sigma^\varphi\overset{j}\hookrightarrow\Tt^{\varphi}$ 
induce the surjective maps 
$\widehat{\Tt^{\varphi}}\overset{^{t}j}\twoheadrightarrow\widehat{\Sigma^\varphi}$ with  
$\Ker\ ^{t}j= \{ f \in \widehat{\Tt^{\varphi}} :\ f(\Sigma^\varphi)=1 \}$ and 
$\widehat{\Tt_{I^c}^{\varphi}}\overset{^{t}j_{s}}\twoheadrightarrow\widehat{\Omega^\varphi}$ 
with $N^\varphi = \Ker\ ^{t}j_{s} = \{f\in\widehat{ \Tt_{I^c}^{\varphi}}:\ f(\Omega^\varphi)=1\}$. 
In particular, we have
\begin{equation}\label{eq:orderSigmaOmega}
|\Sigma^\varphi|=|\Tt_{I}^{\varphi}||\Omega^\varphi| = |\Tt_{I}^{\varphi}| \frac{|{\Tt_{I^c}^{\varphi}}|}{|N^\varphi|}=
|\Tt_{I}^{\varphi}| \frac{|{\Tt^{\varphi}}|}{|\Tt_{I}^{\varphi}||N^\varphi|} = 
\frac{\ell^n}{|N^\varphi|}.
\end{equation}
Moreover, one has that 
$\Ker\ ^{t}(j\iota)\simeq \widehat{\Tt_{I^c}^{\varphi}}$, 
since there is a group monomorphism 
$\Ker\ ^{t}(j\iota)\to \widehat{\Tt_{I^c}^{\varphi}}$ and $| \Ker\ ^{t}(j\iota)|=|\Tt_{I^c}^{\varphi}|$. 
Hence, in what follows we identify the elements of $\widehat{\Tt_{I^c}^{\varphi}}$ and  $\Ker\ ^{t}(j\iota)$.
On the other hand, if we denote $\hat{D}^z=\hat{D}_1^{z_1}\cdots \hat{D}_n^{z_n}$  
for all $z=(z_{1},\ldots,z_{n}) \in (\ZZ / \ell\ZZ)^n$, then
$$\Ker\ ^{t}(j\iota)= \{ \hat{D}^z\vert\ \hat{D}^z\left(\overline{K}_i\right)=1=\hat{D}^z(\widetilde{K}_j),\text{ for }
i\in I_+,\ j\in I_-,\ z\in (\ZZ / \ell\ZZ)^n
\}\simeq  \widehat{\Tt_{I^c}^{\varphi}}.$$ 
Therefore, if $\hat{D}^z \in  \Ker\ ^{t}(j\iota)$, then 
\begin{align}\label{eq:dedkerij1}
1 &= \hat{D}^z\left(\overline{K}_i\right) =  \hat{D}^z (K_{(1-\varphi)(\alpha_i)}) = 
\hat{D}^z \left(K_{\alpha_i}\prod_{j=1}^{n}K_{\alpha_{j}}^{-2y_{ji}}\right)=
\epsilon^{z_{i}}\prod_{j=1}^{n}\epsilon^{-2y_{ji}z_{j}},\\ \label{eq:dedkerij2}
1 &= \hat{D}^z(\widetilde{K}_{j}) =  \hat{D}^z (K_{(1+\varphi)(\alpha_j)})=
\hat{D}^z \left(K_{\alpha_j}\prod_{k=1}^{n}K_{\alpha_{j}}^{2y_{kj}}\right)=
\epsilon^{z_{j}}\prod_{k=1}^{n}\epsilon^{2y_{kj}z_{j}},
\end{align}
for all $i \in I_{+}$ and $j\in I_{-}$. Thus, to find the generators of $\Ker\ ^{t}(j\iota)$ it suffices to 
solve a linear system over $\ZZ/\ell\ZZ$. Indeed, if
$I_+=\{\alpha_{i_1},\ldots, \alpha_{i_s}\}$ and $I_-=\{ \alpha_{j_1},\ldots, \alpha_{j_r}\}$,  by \eqref{eq:dedkerij1} 
and \eqref{eq:dedkerij2} we have a
system of linear equations over  $\ZZ/\ell\ZZ$ whose matrix $S^{\varphi}_{\ell}$ is given by  
\begin{equation*}
\left(\begin{array}{cccccccccccc}
 2y_{i_1 1}&\ldots &2y_{i_1 i_s}&\ldots &2y_{i_1 j_1}& \ldots &2y_{i_1  j_l}&\ldots  &2y_{i_1 i_1}-1&\ldots & 2y_{i_1 n}  \\ 
\vdots & \ddots &  \vdots & \ddots & \vdots  & \ddots &  \vdots & \ddots & \vdots  & \ddots & \vdots \\ 
 2y_{i_s 1} & \ldots  & 2y_{i_s i_s}+1 & \ldots & 2y_{i_s j_1} & \ldots  & 2y_{i_s j_l} & \ldots & 2y_{i_s i_1} & \ldots & 2y_{i_s n} \\ 
\vdots & \ddots &  \vdots & \ddots & \vdots  & \ddots &  \vdots & \ddots & \vdots  & \ddots & \vdots \\ 
 2y_{j_1 1} & \ldots  & 2y_{j_1 i_s}  & \ldots & 2y_{j_1 j_1}-1  & \ldots & 2y_{j_1 j_l} & \ldots  & 2y_{j_1 i_1} & \ldots & 2y_{j_1 n} \\ 
\vdots & \ddots &  \vdots & \ddots & \vdots  & \ddots &  \vdots & \ddots & \vdots  & \ddots & \vdots \\ 
 2y_{j_l 1} & \ldots  & 2y_{j_l i_s}  & \ldots & 2y_{j_l j_1} & \ldots & 2y_{j_l j_l}+1 & \ldots  & 2y_{j_l i_1} & \ldots & 2y_{j_l n}
 \end{array}\right)
\end{equation*}
In particular, $|\Ker\ ^{t}(j\iota)|=| \widehat{\Tt_{I^c}^{\varphi}}|=\ell^{n-\rk S^{\varphi}_{\ell}}$.
Analogously, it is possible to characterize in the same way the kernel $N^\varphi$. In this case we 
have to consider  the system of linear equations  determined by the conditions  $\hat{D}^{z}(\Omega^\varphi)=1$ 
for all  $ \hat{D}^z \in\widehat{\Tt_{I^c}^{\varphi}}$. 
\end{obs}

\begin{exa}\label{exa:1}Assume $\lieg$ is of type $C_3$ with associated Cartan matrix $A = 
\left(\begin{smallmatrix}
   2 & -1 & 0\\
   -1 & 2 & -1\\
   0 & -2 & 2
\end{smallmatrix}\right)$. Then the multiparametric matrix $Y$ is given by
$$Y=\begin{pmatrix}
    a+b/2 & -a+c/2 & -b/2-c/2\\
   2a+b & -a+c & -b/2-c\\
   2a+3b/2 & -a+3c/2 & -b/2-c
\end{pmatrix},$$
where $a\in \ZZ,$ and $b,c\in 2\ZZ$.
Set $a=1$, $b=2$, $c=0$ and $\ell=11$. Then,   
$\varphi(\alpha_1)=4\alpha_1
 +8\alpha_2+10\alpha_3$,  
 $\varphi(\alpha_2)=-2\alpha_1
 -2\alpha_2-2\alpha_3$ and 
 $\varphi(\alpha_3)=-2\alpha_1
 -2\alpha_2-2\alpha_3$.
 
 $(a)$ If we choose $I_+=\{\alpha_2\}$ and
 $I_-=\{\alpha_1\}$, then $S^{\varphi}_{11}= 
 \left(\begin{smallmatrix}
    5 & 8 & 10\\
   2 & 3 & 2
\end{smallmatrix}\right) \sim_{11} 
\left(\begin{smallmatrix}
    1 & 0 & 8\\
   0 & 1 & 10
\end{smallmatrix}\right)$ and 
$\widehat{\Tt^{\varphi}_{I^{c}}}=\langle \hat{D}_1^3\hat{D}_2\hat{D}_3 \rangle\simeq \ZZ/11\ZZ$. 
 If we take $\Sigma^\varphi=\langle K_{(1-\varphi)(\alpha_2)}, K_{(1+\varphi)(\alpha_1)},  K_{\tau_3}, K_{\tau_2}\rangle$, 
then we have that $\Sigma^\varphi = \Tt^{\varphi} \simeq (\ZZ/11\ZZ)^{3}$ and $N^\varphi$ is trivial.

\smallbreak
$(b)$ If we choose $I_{+}=\{\alpha_2\}$, $I_-=\emptyset$ and 
$\Sigma^\varphi=\langle K_{(1+\varphi)(\alpha_1)}, K_{(1-\varphi)(\alpha_2)}  \rangle$, 
then we have that $\widehat{\Tt^\varphi_{I^c}}= \langle \hat{D}^{(1,0,10)}, \hat{D}^{(0,1,4)}\rangle\simeq (\ZZ/11\ZZ)^{2}$, 
$ \Omega^\varphi \simeq \langle K_{(1+\varphi)(\alpha_1)}\rangle$ and 
$N^\varphi=\langle \hat{D}^{(3,1,1)} \rangle$.
\end{exa}

The following proposition states that the elements in $ \widehat{\Tt_{I^c}^{\varphi}}$ are central in
$\qephip^\ast$.

\begin{prop}\label{prop:dzcen}
The subgroup $ \widehat{\Tt_{I^c}^{\varphi}}$ of $G(\qephip^\ast)$ consists of central group-like elements.
\end{prop}

\pf
Let $z\in (\ZZ/\ell\ZZ)^{m}$ and $D^z\in G(\qephip^\ast)$ such that $D^{z} \in \widehat{\Tt_{I^c}^{\varphi}}$. Then 
$D^z(\overline{K}_i)=1=D^z(\widetilde{K}_j)$ 
for all $i\in I_+$ and $j \in I_-$. We show that $D^z$ is central in $\qephip^\ast$.

By \cite[Theorem 6.7]{L2} and \cite[Lemma 2.14]{AG}, $\qephip$ has a basis 
$$
\left\{\prod_{\beta\geq0}F_\beta^{n_\beta}\prod_{i=1}^{n}K_{\alpha_i}^{t_i}
\prod_{\alpha\geq0}E_\alpha^{m_\alpha}: 0 \leq n_\beta, t_i, m_\alpha \leq \ell,\text{ } 
1 \leq i \leq n,\text{ } \beta \in Q_{I_-},\text{ } \alpha \in Q_{I_+} \right\}.
$$
The hypotesis on $D^{z}$ ensures that $D^{z}f(E_{i}) = fD^{z}(E_{i})$ and 
$D^{z}f(F_{j}) = fD^{z}(F_{j})$ for all $f\in \qephip^\ast$, $z\in (\ZZ/\ell\ZZ)^{m}$, $i\in I_+$ and $j \in I_-$.
Moreover, since the elements $K_{\alpha_{t}}\in \qephip$ are group-like for all $1\leq t\leq n$, 
$D^{z}f(K_{\alpha_{t}}) = fD^{z}(K_{\alpha_{t}})$. As $D^{z}$ is a group-like element in $\qephip^\ast$,
we have that $D^{z}f(MN) = fD^{z}(MN)$ for $M,N\in \{K_{\alpha_{t}}, E_{i}, F_{j}: i\in I_{+}, j\in I_{-}\}$, since
$D^{z}f(MN) = (D^{z}f)_{(1)}(M)(D^z f)_{(2)}(N) = D^{z}f_{(1)}(M)D^z f_{(2)}(N) 
= f_{(1)}D^{z}(M)f_{(2)}D^{z}(N) = fD^{z}(MN)$.
Analogously,  using an inductive argument
one may prove that $D^{z}$ and $f$ commute 
when evaluated on every element of the basis.
\epf

The following proposition gives a characterization of all quotients of $\qephi^{*}$. 

\begin{prop}\label{prop:propdih} Let $H$ be a Hopf algebra quotient of $\qephi^*$ such that $H^{*}$ is
determined by the triple $(I_{+}, I_{-},\Sigma^{\varphi})$ and $\qephip$ the twisted regular Frobenius-Lusztig kernel
associated to $(I_{+},I_{-})$. Then
$H=\qephip^\ast/\langle D^z-1 :\ D^{z} \in N^\varphi \rangle$.
\end{prop}

\begin{proof}
If $(I_+\cap I_-)^c = \{\alpha_{i_{1}},\ldots, \alpha_{i_{m}}\}$ and we write 
$D^{z} = D_{i_{1}}^{z_{1}}\cdots D_{i_{m}}^{z_{m}}$, then  Remark \ref{rmk:notathenDi} $(b)$, $G(\qephip^*)=
\{ D^{z}\vert\ z\in (\ZZ/\ell\ZZ)^{m}\}$.
By Proposition \ref{prop:dzcen}, we know that the elements of $\widehat{\Tt^{\varphi}_{I^{c}}}$ are central in $\qephip^*$. 
Since $N^{\varphi} \subseteq  \widehat{\Tt^{\varphi}_{I^{c}}}$, the two-sided ideal $\II$ of $\qephip^*$
generated by the elements $ \{D^z-1 :\ D^{z} \in N^\varphi\} $ is a Hopf ideal 
and whence $\qephip^*/\II$ is a Hopf algebra. 

On the other hand, we know that $H^{*}$ is determined by the triple $(I_{+}, I_{-},\Sigma^{\varphi})$, 
and consequently, $H^{*}$ is included in $\qephip$. If we denote by $\nu: \qephip^{*} \to H$ the epimorphism 
induced by this inclusion, we have that $\Ker \nu = \{f \in \qephip^{*} :\ f (h) = 0\text{ for all }h \in H^{*} \}$.
Since by Remark \ref{rmk:notathenDi} $(c)$, $D^{z}(g) = 1$ for all $g\in \Sigma^{\varphi}$ and 
$D^{z} \in N^{\varphi}$, we have that $D^{z}-1 \in \Ker \nu$ and whence 
there is a Hopf algebra epimorphism $\qephip^\ast/\II \twoheadrightarrow H$. 
But by \eqref{eq:orderSigmaOmega} we have that 
$$\dim H = |\Sigma^{\varphi}|\ell^{|I_{+}|+|I_{-}|}= \frac{\ell^{n}}{|N^{\varphi}|}\ell^{|I_{+}|+|I_{-}|} = 
\dim \qephip^{*}/\II,$$
which implies that the epimorphism is indeed an isomorphism.
\end{proof}

\begin{exa}\label{exa:2}
 Let $\varphi$ be the twisting map defined in Example \ref{exa:1} over $\g=\mathfrak{sp}_{6}$. 
 If we take $I_+=\{\alpha_2\}$,
 $I_-=\{\alpha_1\}$ and $\Sigma^\varphi=\langle K_{(1-\varphi)(\alpha_2)}, K_{(1+\varphi)(\alpha_1)},  K_{\tau_3}, K_{\tau_2}\rangle$, 
then $\Sigma^\varphi = \Tt^{\varphi} \simeq (\ZZ/11\ZZ)^{3}$ and $N^\varphi$ is trivial. On the other hand, if we set
$\varphi=0$, then $\Sigma= \langle K_{\alpha_{1}},K_{\alpha_{2}}\rangle$ and $N$ is not trivial. This implies that 
the quotient $\qephip^\ast/\langle D^z-1 :\ D^{z} \in N^\varphi \rangle$ cannot be a $2$-cocycle deformation
of $\qep^\ast/\langle D^z-1 :\ D^{z} \in N \rangle$, since they have different dimension.

\end{exa}

\section{Quantum subgroups}\label{sec:qsubgroups}
In this section we determine all quantum subgroups of the twisted quantum group $G_{\eps}^{\varphi}$.
We first construct a family of quantum subgroups using the root datum associated to $\lieg= \Lie(G)$
and an algebraic subgroup $\Gamma$ of $G$. Then we prove that any quantum subgroup of $G_{\eps}^{\varphi}$
is isomorphic to one constructed in this way. We end the section with a parametrization of the isomorphism classes.

From now on, we work with the complex form of all quantum algebras introduced above. 

\subsection{Twisted quantum regular subgroups}
Let $I_{\pm} \subseteq {\pm}\Pi$. Let $\Gephip$ be the Hopf algebra associated to the pair $(I_{+}, I_{-}) $ as in 
Definition \ref{def:gamaparabo}, and $\liep$ the regular Lie subalgebra of $\lieg$ given by Remark \ref{rmk:liep}.
In this subsection we construct the twisted quantum function algebras related to the pair $(I_{+}, I_{-}) $.

Denote by $\Res: \Gphige^\circ\to\Gephip^\circ$ the Hopf algebra map induced 
by the inclusion $\Gephip\hookrightarrow\Gphige$. Using Lemma \ref{lem:pairind}, we know that 
$\Oephi\subseteq\Gphige^\circ$. 

\begin{definition} We define the twisted quantum function algebra associated to the 
regular Lie subalgebra $\p$ of $\lieg$ as the Hopf algebra given by
$$\OephiP:=\text{Res}(\Oephi).$$
\end{definition}
If $\varphi=0$, we have that $\OO^{0}_{\epsilon}(P) = \OeP$, see \cite[\S 2.3.1]{AG}.
Since $\Oc(G) $ is a central Hopf subalgebra of $\Oephi $, 
$\Res (\Oc(G) )$ is a central Hopf subalgebra of 
$\OephiP $. Thus, there exists $P$ an algebraic subgroup of $G$ such that 
$\Res (\Oc(G) ) = \Oc(P) $. Since $\Oc(P) $ is a 
central Hopf subalgebra of $\OephiP $, the quotient 
$$ \overline{\OephiP}  := \OephiP  / 
[\Oc(P) ^{+}\OephiP ],$$ 
is a Hopf algebra, which is in fact isomorphic to $\qephip^{*}$. 

\begin{prop}\label{prop:paraL}
\begin{enumerate}
\item[$(i)$] $P$ is a connected algebraic group and $\Lie (P)=\p$.
\item[$(ii)$] The following sequence of Hopf algebras is exact 
$$
1 \longrightarrow \OO(P)  \longrightarrow \OephiP  \longrightarrow 
\overline{\OephiP}  \longrightarrow 1.
$$
\item[$(iii)$] There exists a Hopf algebra epimorphism $\overline{\Res}: \qephi^{*} \to \overline{\OephiP} $
making the following diagram commutative
\begin{equation}\label{eq:diaP}
\xymatrix{1 \ar[r]^{} & \OO(G) \ar[r]^{\iota}\ar[d]^{\res} &
\Oephi  \ar[r]^{\pi}\ar[d]^{\Res}
& \qephi^\ast \ar[r] \ar[d]^{\overline{\Res}}& 1\\
1 \ar[r]^{} & \OO(P) \ar[r]^{\iota_P} &
\OephiP  \ar[r]^{\pi_P}
& \overline{\OephiP}   \ar[r] & 1.}
\end{equation}
\item[$(iv)$] $\OephiP  $ and $\overline{\OephiP} $ 
are $2$-cocycle deformations of $\OeP $ and 
$\overline{\OeP} $, respectively.
\item[$(v)$] $\overline{\OephiP} \simeq\qephip^\ast$ as Hopf algebras. 
\end{enumerate}
\end{prop}

\pf $(i)$, $(ii)$, $(iii)$
follow \textit{mutatis mutandis} from \cite[Propositions 2.7 \& 2.8]{AG}.

$(iv)$ By Lemma \ref{lem:def2cocicloOe}, we know that $\Oephi $ 
is a $2$-cocycle
deformation of $\Oe $, say by the cocycle $\bar{\sigma}$. 
Since the kernel $\II$ of the Hopf algebra map $\Res: \Oe  \to \OeP $ 
is spanned by matrix coefficients that 
vanish when restricted to $\Gamma_\epsilon(\p)$, using the definition of $\bar{\sigma}$
we see that $\bar{\sigma}|_{\II\ot \Oe + \Oe \ot \II}= 0$. Thus by Remark \ref{obs:cocycleproy},
$\Res$ induces a $2$-cocycle $\hat{\sigma}$ on $\Oe /\II$ and 
we have that $\OephiP = \Res((\Oe )_{\bar{\sigma}}) = (\Oe /\II)_{\hat{\sigma}}=
(\OeP )_{\hat{\sigma}}$. The same argument applies for $\overline{\OephiP} $
and $\overline{\OeP} $, since 
$\OO(P) $ is a central Hopf subalgebra of $ \OeP $ and 
the cocycle $\hat{\sigma}$ is trivial on it.

$(v)$ Dualizing the diagram \eqref{eq:diapara} we get 
$$
\xymatrix{1 \ar[r]^{} & U(\g)^{\circ}  \ar@{^{(}->}[r]^{^{t}\Fr} 
\ar@{->>}[d] &
\Gphige^{\circ}\ar[r]^{\alpha} \ar@{->>}[d]^{\Res}
& \qephi^{*} \ar[r] \ar@{->>}[d] & 1\\
1 \ar[r]^{} & U(\p)^{\circ}  \ar@{^{(}->}[r]^{^{t}\Fr_{res}}  &
\Gephip^{\circ} \ar[r]^{\beta}
& \qephip^{*}  \ar[r] & 1.}
$$
Since $\lieg$ is simple, we have that $\OO(G)  \simeq U(\g)^{\circ} $.
Thus, as $\OO(P)  = \Res (\OO(G) )$ and 
$\OephiP  = \Res (\Oephi )$, we have that 
$^{t}\Fr_{res}(\OO(P) ) \subseteq U(\p)^{\circ} $ and consequently
$\OO(P) ^{+} \subseteq \Ker \beta$. Moreover, since $\alpha(\Oephi ) 
= \pi (\Oephi )=\qephi^{*}$, we have that 
$\qephip^{*} = \beta\Res(\Oephi ) = \beta(\OephiP )$. Hence, 
there exists a surjective Hopf algebra map $\gamma:\overline{\OephiP}   \to \qephip^{*}$.
But by $(iv)$,  \cite[Proposition 2.8 (c)]{AG} and Proposition \ref{prop:uptup}, we have that 
$\dim \overline{\OephiP} =\dim \overline{\OeP}  = \dim \qep= \dim \qephip$ and 
the epimorphism is in fact an isomorphism.
\epf

\begin{obs}\label{rmk:2cocycleOephiP} 
By the proposition above, we know that $\OephiP$ fits into the central exact sequence
of Hopf algebras $
\xymatrix{\OO(P)\ar@{^{(}->}[r]^{\iota_{P}} & \OephiP  \ar@{->>}[r]^{\pi_{P}} & \qephip^{*}} $ and 
that $\OephiP$ is a $2$-cocycle deformation of $\OeP$, where the $2$-cocycle $\hat{\sigma}$
is given by the formula $\hat{\sigma}(\Res(x),\Res(y)) = \bar{\sigma}(x,y)$ for all 
$x,y\in \Oe$. On the other hand, by Propositions \ref{prop:froblustkertwist} and \ref{prop:uptup} we know that 
$\qephip^{*} = (\qep^{*})_{\tau}$
for the $2$-cocycle $\tau$ given by $\tau(\overline{\Res}(\pi(x)),\overline{\Res}(\pi(y)) = 
\bar{\sigma}(x,y)$. Since the diagram \eqref{eq:diaP} for $\varphi = 0$ is commutative, 
the pullback of the cocycle $\tau$ coincides with the cocycle $\hat{\sigma}$.
\end{obs}

\subsection{Quantum subgroups from classical subgroups}
In this subsection we construct a Hopf algebra quotient of $\Oephi$ associated to the pair $(I_{+},I_{-})$ 
and an algebraic
subgroup of $G$ included in $P$. 
This is based in the  \textit{pushout construction}, 
which is a general method for constructing Hopf algebras from central
exact sequences. 

The following proposition follows from the arguments in \cite[\S 2.2]{AG}. If 
$\gamma:\Gamma\to G$ is a homomorphism of algebraic groups, then 
$^{t}\gamma:\OO(G) \to \OO(\Gamma)$ denotes the corresponding algebra map
between the coordinate algebras.

\begin{prop}\label{prop:pushout}
Let $\Gamma$ be an algebraic group and $\gamma:\Gamma\to G$ an
injective homomorphism of algebraic groups such that $\gamma(\Gamma)\subseteq P$. 
Let $\JJ$ denote the two-sided ideal of $\OephiP$ generated by $\iota(\Ker\ ^{t}\gamma)$.
Then
$ A^\varphi_{\epsilon,\mathfrak{p},\gamma} = \OephiP/\JJ $ is a Hopf algebra and there exist 
a Hopf algebra monomorphism 
$j:\OO(\Ga) \hookrightarrow   A^\varphi_{\epsilon,\mathfrak{p},\gamma}$, and Hopf algebra epimorphism 
$\bar{\pi}: A^\varphi_{\epsilon,\mathfrak{p},\gamma} \twoheadrightarrow \qephip^{*} $ such that 
$A^\varphi_{\epsilon,\mathfrak{p},\gamma}$ fits into the exact sequence of Hopf algebras 
$$\xymatrix{1 \ar[r]^{} & \OO(\Gamma) \ar[r]^{j} & A^\varphi_{\epsilon,\mathfrak{p},\gamma}\ar[r]^{\overline{\pi}} & 
\qephip^{*}\ar[r]^{} & 1.}
$$
If in addition $|\Ga|$ is finite, then $\dim A^\varphi_{\epsilon,\mathfrak{p},\gamma} = 
|\Ga| \dim \qephip$.
Moreover, the following diagram is commutative
\begin{equation}\label{diag:pushout}
\xymatrix{1 \ar[r]^{} & \Oc(G) \ar[r]^{\iota} \ar[d]_{\res} &
\Oephi \ar[r]^{\pi}
\ar[d]^{\Res} & \qephi^{*} \ar[r]^{}\ar[d]^{\overline{\Res}} & 1\\
1 \ar[r]^{} &   \Oc(P)  \ar[r]^{\iota_P}\ar[d]_{^t\gamma} &
\Oc_\epsilon^\varphi(P) \ar[r]^{\pi_P}\ar[d]^{\psi} & \qephip^{*} \ar[r]^{}\ar[d]^{\id} & 1\\
1 \ar[r]^{} & \OO(\Gamma) \ar[r]^{j} & A^\varphi_{\epsilon,\mathfrak{p},\gamma}\ar[r]^{\overline{\pi}} & 
\qephip^{*}\ar[r]^{} & 1.} 
\end{equation}
\qed
\end{prop}

\begin{prop}\label{prop:pushoutcocycle}
$A^\varphi_{\epsilon,\mathfrak{p},\gamma}$ is a $2$-cocycle deformation of ${A_{\epsilon,\mathfrak{p},\gamma}}$. 
\end{prop}

\pf By Proposition \ref{prop:paraL} $(iv)$, we know that $\OephiP$ is a $2$-cocycle deformation 
of $\OeP$, say by the cocycle $\hat{\sigma}$, see Remark \ref{rmk:2cocycleOephiP} above. Then,
by Remark \ref{obs:cocycleproy} it is enough to check that $\hat{\sigma}|_{\OephiP\ot\JJ + \JJ\ot \OephiP} = 0$. 
Since $\JJ= \OephiP \iota_{P}(\Ker\ ^{t}\gamma)$ and $\Ker\ ^{t}\gamma$ is 
generated by matrix coefficients $c_{f,v}$ in $\OO(P)$, of degree $(\ell \lambda, \ell \mu) $
for some $\lambda, \mu \in P$, we have that 
$\hat{\sigma}|_{\iota_{p}(\Ker\ ^{t}\gamma)\ot \iota_{p}(\Ker\ ^{t}\gamma)} = \eps\ot\eps = 0$ and whence
$\hat{\sigma}|_{\OephiP\ot\JJ + \JJ\ot \OephiP} = 0$. 
Thus, we may define a $2$-cocycle 
$\tilde{\sigma}:A_{\epsilon,\mathfrak{p},\gamma}\ot A_{\epsilon,\mathfrak{p},\gamma} \to \CC$
by $\tilde{\sigma}(\psi(x),\psi(y)) = \hat{\sigma}(x,y)$ for all $x,y \in \OeP$ and $(A_{\epsilon,\mathfrak{p},\gamma})_{\tilde{\sigma}} = 
A^\varphi_{\epsilon,\mathfrak{p},\gamma}$.
Note that $\tilde{\sigma}$ coincides with the pullback through $\bar{\pi}$ of the $2$-cocycle $\tau$ on $\qep^{*}$.
\epf

By Proposition \ref{prop:qephisse} $(ii)$, we know that there exists an injective
coalgebra map $^{t}\theta: \qephip^\ast \to \Gamma_\epsilon^\varphi(\p)^\circ$, and since $\qephip^{*} \simeq \overline{\OephiP}$ by Proposition \ref{prop:paraL},
we have that $\Img\ ^{t}\theta\subseteq \OephiP$. Thus, the image of the central subgroup $\widehat{\Tt_{I^c}^{\varphi}}$ of $G(\qephip^{*})$ is a 
subgroup of $G(\OephiP)$. Denote $d^{z} = \ ^{t}\theta(D^{z})$ for $z\in (\ZZ/\ell\ZZ)^{m}$, $D^{z}\in \widehat{\Tt_{I^c}^{\varphi}}$.

\begin{lema}\label{lem:centralgrelements}
There exists a subgroup $\mathcal{A}= \{\partial^{z}=\psi(^{t}\theta(D^{z})):\ D^{z} \in  \widehat{\Tt_{I^c}^{\varphi}}\}$ of 
$G(A^\varphi_{\epsilon,\mathfrak{p},\gamma})$ isomorphic to $\widehat{\Tt_{I^c}^{\varphi}}$ consisting of central elements. 
In particular, $|\mathcal{A}|=\ell^{n-\rk S^{\varphi}_{\ell}}$.
\end{lema}

\pf
Using the same argument as in the proof of Proposition \ref{prop:dzcen},
one sees that the elements $d^{z}$ are central in $\OephiP$. Indeed, if $f\in \OephiP$, then 
$d^{z}f(M) = fd^{z}(M)$ for every generator $M$ of $\Gephip$ from Definition \ref{def:gamaparabo}. For example, let $i\in I_{+}$ and 
$m\geq 0$, then by \eqref{eq:coproGphi} we have 
\begin{align*}
d^{z}f\left(E_{i}^{(m)}\right) &= \sum\limits_{r+s=m}q_i^{-rs}d^{z}\left(E_i^{(r)}K_{s(\alpha_i-\tau_{i})}\right) f\left(E_i^{(s)}K_{r\tau_i}\right)\\
&= \sum\limits_{r+s=m}q_i^{-rs}d^{z}\left(E_i^{(r)}\right)d^{z}(K_{s(\alpha_i-\tau_{i})}) f\left(E_i^{(s)}K_{r\tau_i}\right)=  d^{z}(K_{m(\alpha_i-\tau_{i})}) f\left(E_i^{(m)}\right),\text{ and }\\
f d^{z}\left(E_{i}^{(m)}\right) &= \sum\limits_{r+s=m}q_i^{-rs}f\left(E_i^{(r)}K_{s(\alpha_i-\tau_{i})}\right) d^{z}\left(E_i^{(s)}K_{r\tau_i}\right)\\
&= \sum\limits_{r+s=m}q_i^{-rs}f\left(E_i^{(r)}K_{s(\alpha_i-\tau_{i})}\right) d^{z}\left(E_i^{(s)}\right)d^{z}(K_{r\tau_i})=  f\left(E_i^{(m)}\right)d^{z}(K_{m\tau_i}).
\end{align*}
Since $d^{z}\left(\overline{K_{i}}\right) = d^{z}(K_{\alpha_{i}-2\tau_{i}})= D^{z}(\theta(K_{\alpha_{i}-2\tau_{i}}))= 1$,
we have that $d^{z}(K_{m(\alpha_i-\tau_{i})}) = d^{z}(K_{m\tau_i})$ for all $m\geq 0$, and then $d^{z}f\left(E_{i}^{(m)}\right) = fd^{z}\left(E_{i}^{(m)}\right)$.
Analogously, using that $ 1= d^{z}\left(\widetilde{K_{j}}\right) $ for all $j\in I_{-}$, we have that  $d^{z}f\left(F_{j}^{(m)}\right) = fd^{z}\left(F_{j}^{(m)}\right)$
 for all $m\geq 0$. The equality on the generators 
 $K^{-1}_{\alpha_i}$ and $
\dbinom{K_{\alpha_i};0}{m}$ follows easily since the coproduct is cocommutative on them.
Applying an inductive argument on monomials on the generators we have that $d^{z}$ is central in $\OephiP$.
Since $\psi: \OephiP \to A^{\varphi}_{\epsilon, \liep,\gamma}$ is surjective, the group-like elements $\partial^{z}$
are also central in $A^{\varphi}_{\epsilon, \liep,\gamma}$.

Now we show that $\mathcal{A}\simeq \widehat{\Tt_{I^c}^{\varphi}}$ as groups. By construction, we have that 
$\psi\circ\ ^{t}\theta: \widehat{\Tt_{I^c}^{\varphi}} \to \mathcal{A}$ is a group epimorphism. As 
the diagram
$$
\xymatrix{\Oc_\epsilon^\varphi(P) \ar[d]_{\psi}\ar[r]^{\pi_P}  & \qephip^\ast \\
 A^\varphi_{\epsilon,\mathfrak{p},\sigma}  \ar[ru]^{\overline{\pi}} & }
$$
 is commutative
 by \eqref{diag:pushout}, we have that  
$\overline{\pi}(\mathcal{A}) =\overline{\pi}(\psi (^{t}\theta(\widehat{\Tt_{I^c}^{\varphi}})) = \pi_P (^{t}\theta(\widehat{\Tt_{I^c}^{\varphi}}) = 
\widehat{\Tt_{I^c}^{\varphi}}
$, which implies that $\psi\circ\ ^{t}\theta$ is indeed an isomorphism.
\epf

\subsection{Quantum subgroups from subalgebras of the twisted Frobenius-Lusztig kernels}
In this subsection we construct Hopf algebras from a Hopf subalgebra of $\qephi$ and 
an algebraic subgroup of $G$.

Let $L$ be a Hopf subalgebra of $\qephi$. By Lemma \ref{lema:subalgparadeuphi}, it is determined by a triple 
$(I_{+}, I_{-},\Sigma^{\varphi})$ with $\Sigma
^\varphi$ a subgroup of $G(\qephi)$ and $I_\pm \subset \pm\Pi$ are such that  $K_{(1\mp\varphi)(\alpha_i)} \in \Sigma^\varphi$
if $\alpha_i\in I_{\pm}$. If $H= L^{*}$, then by Proposition \ref{prop:propdih}, 
$H = \qephip^{*}/ \langle D^z-1 :\  D^{z} \in N^\varphi \rangle$, 
where $N^\varphi$ is determined by $\Sigma^{\varphi}$ as 
in Remark \ref{rmk:notathenDi} $(d)$. Let $P$ be the regular subgroup of $G$
determined by the pair $( I_{+}, I_{-})$ with $\liep = \Lie(P)$ and $\OephiP$, $\qephip$
the corresponding twisted quantum algebras.

\begin{prop}\label{prop:ultimopaso}
Let $\Gamma$ be an algebraic group and $\gamma:\Gamma \to G$ an injective morphism
of algebraic groups such that $\gamma (\Ga) \subseteq P $. For every group homomorphism
$\delta: N^{\varphi}\to \widehat{\Ga}$, the two-sided ideal $J_{\delta}$ of $A^{\varphi}_{\epsilon, \liep,\gamma}$
generated by the elemens $\delta(D^{z}) - \partial^{z}$ for $\partial^{z}\in \mathcal{A}$ and $D^{z}$
in $N^{\varphi}$, is a Hopf ideal and the Hopf algebra $A_{\epsilon, \liep,\gamma} / J_{\delta}$
fits into the central exact sequence 
$$1 \longrightarrow \OO(\Gamma) \overset{\hat{\iota}}{\longrightarrow} 
A^{\varphi}_{\epsilon, \p,\gamma} / J_{\delta}  \overset{\hat{\pi}}{\longrightarrow}
 H \longrightarrow  1.$$
If in addition $|\Ga|$ is finite, then $\dim A^\varphi_{\epsilon,\mathfrak{p},\gamma} / J_{\delta} = 
|\Ga| \dim H$.
Moreover, the following diagram is commutative 
\begin{equation}\label{diag:ultimopaso}
\xymatrix{1 \ar[r]^{} & \Oc(G) \ar[r]^{\iota} \ar[d]_{\res} &
\Oephi \ar[r]^{\pi}
\ar[d]^{\Res} & \qephi^{*} \ar[r]^{}\ar[d]^{\overline{\Res}} & 1\\
1 \ar[r]^{} &   \Oc(P)  \ar[r]^{\iota_P}\ar[d]_{^{t}\gamma} &
\Oc_\epsilon^\varphi(P) \ar[r]^{\pi_P}\ar[d]^{\psi} & \qephip^{*} \ar[r]^ {}\ar[d]^{\id} & 1\\
1 \ar[r]^{} & \OO(\Gamma) \ar[r]^{j} \ar[d]_{\id}& 
A^\varphi_{\epsilon,\mathfrak{p},\gamma}\ar[r]^{\overline{\pi}} \ar[d]^{}& \qephip^{*}\ar[r]^{}\ar[d]^{\nu} & 1\\
1 \ar[r] & \OO(\Gamma) \ar[r]^{\hat{\iota}} &
A^{\varphi}_{\epsilon, \liep,\gamma} / J_{\delta}\ar[r]^{\hat{\pi}}
& H \ar[r] & 1.}
\end{equation}
\end{prop}

\pf Follows by the proof \cite[Theorem 2.17]{AG}. We reproduce the first part here to give an idea.
By Lemma \ref{lem:centralgrelements}, we know that the group-like elements $\partial^{z} \in \mathcal{A}$  are central in 
$A^{\varphi}_{\epsilon, \liep,\gamma}$. Since $\delta(D^{z}) \in \OO(\Ga)$ for all $D^{z}\in N^{\varphi}$, the ideal $J_{\delta}$ in 
$A^{\varphi}_{\epsilon, \liep,\gamma}$ generated by the elements $\delta(D^{z}) - \partial^{z}$
is a Hopf ideal, and whence $A^{\varphi}_{\epsilon, \liep,\gamma}/J_{\delta}$ is a Hopf algebra. 
If we write $\JJ_{\delta} = J_{\delta}\cap \OO(\Ga)$,
then $A^{\varphi}_{\epsilon, \liep,\gamma}/J_{\delta}$ fits into the central exact sequence
$$1 \longrightarrow \OO(\Gamma)/\JJ_{\delta} \longrightarrow
A^{\varphi}_{\epsilon, \p,\gamma} / J_{\delta}\longrightarrow \qephip^{*}/\overline{\pi}(J_{\delta}) \longrightarrow 1.$$
Since $\overline{\pi}(\delta(D^{z}))=1$ and 
$\overline{\pi}(\partial^{z})= D^{z}$ by the proof of Lemma \ref{lem:centralgrelements},
it follows that 
$\overline{\pi}(J_{\delta}) = \langle D^z-1 :\ D^{z} \in N^\varphi \rangle$. Thus,  
by Proposition \ref{prop:propdih} $(ii)$, we have that  $\qephip^{*}/\overline{\pi}(J_{\delta}) = H $.
The proof that $\JJ_{\delta} =0$ and that $A^{\varphi}_{\epsilon, \liep,\gamma}/J_{\delta}$ fits into the 
commutative diagram follow the same arguments used in \textit{loc.\ cit}.
\epf

\subsection{Parametrization of quantum subgroups} 
In this subsection we parametrize the Hopf algebra quotients of $\Oephi$ by a $6$-tuple called 
\textit{twisted subgroup datum}. We show first that 
there is a $1$-$1$ correspondance between 
Hopf algebra quotients of $\Oephi$ and twisted subgroup data, and then we classify these quotients  
up to isomorphism.

\begin{definition}\label{def:twistedsdata} A \textit{twisted subgroup datum} 
is a collection $\D^\varphi:= (I_+,I_-,N^\varphi,\Gamma,\gamma,\delta)$ where
\begin{itemize}
 \item[$\triangleright$]  $I_\pm\subset \pm\prod$. Let $\Psi_{\pm} = \{\alpha \in \Phi: \Supp \alpha \in I_{\pm}\}$, 
 $\liep = \sum_{\alpha \in \Psi_{\pm}} \g_{\alpha}$ and $\liep = \liep_{+ } \oplus \lieh \oplus \liep_{-}$. 
 Let $P$ be the connected Lie subgroup of $G$ with
 $\Lie (P) = \liep$.
\item[$\triangleright$]  $N^\varphi$ a subgroup of $\widehat{\Tt^{\varphi}_{I^{c}}}$, see Remark \ref{rmk:notathenDi} $(d)$. 
\item[$\triangleright$] $\Gamma$ an algebraic group.
\item[$\triangleright$]   $\gamma:\Gamma\to 
P$ is a injective algebraic group homomorphism.
\item[$\triangleright$] $\delta:N^\varphi\to \widehat{\Gamma}$ is a group homomorphism.
\end{itemize}
If $\Ga$ is finite, we call $\D^{\varphi}$ a \textit{finite twisted subgroup datum}.
\end{definition} 

Summarizing the previous results we obtain the first main result of the paper.
\begin{theorem}\label{thm:mainteo1}
Let $\D^\varphi = (I_+,I_-,N^\varphi,\Gamma,\gamma,\delta)$ 
be a twisted subgroup datum. 
Then there exists a Hopf algebra $A_{\D^{\varphi}}= A_{\epsilon, \liep,\gamma} / J_{\delta}$  
of $\Oephi$ that fits into the central exact sequence 
$$
\xymatrix{1 \ar[r] & \OO(\Gamma) \ar[r]^{\hat{\iota}} &
A_{\D^{\varphi}}\ar[r]^{\hat{\pi}}
& H \ar[r] & 1.}
$$
In particular, if $|\Ga|$ is finite, then 
 $\dim A_{\D^{\varphi}} = |\Ga| \dim H$.
\end{theorem}

\begin{proof} 
By Lemma \ref{lema:subalgparadeuphi} and Remark \ref{rmk:notathenDi} $(d)$, the triple $(I_+,I_-,N^\varphi)$
determines a quotient $H$ of $\qephi^{*}$. Besides, by Proposition \ref{prop:paraL}, the pair $(I_+,I_-)$ determines
a regular subgroup $P$ of $G$, a regular Lie subalgebra $\liep$ of $\g$ and the quantum algebras $\OephiP$ 
and $\qephip$, which makes the upper part of the diagram \eqref{diag:ultimopaso} commutative. 
Then by Proposition \ref{prop:pushout}, the morphism 
$\gamma:\Gamma \to P \subset G$ give rise to the Hopf algebra 
$A_{\epsilon,\p,\gamma}^\varphi$ through the pushout construction. 
Finally, by Proposition \ref{prop:ultimopaso} the 
group homomorphism $\delta: N^{\varphi} \to \widehat{\Ga}$ defines the Hopf 
ideal $J_{\delta}$ of $A_{\epsilon,\p,\gamma}^\varphi$ and the Hopf 
algebra $A_{D^{\varphi}} = 
A_{\epsilon,\p,\sigma}^\varphi / J_{\delta}$ fits into the commutative diagram \eqref{diag:ultimopaso}.
\end{proof}

The next theorem establishes the converse of Theorem \ref{thm:mainteo1}. We give its 
proof in several lemmata.

\begin{theorem}\label{thm:reciteo}
Let $\kappa :\Oephi\to A$ be a surjective Hopf algebra morphism, then there exists a twisted
subgroup datum $\D^\varphi$ such that $A\simeq A_{\D^{\varphi}}$ as Hopf algebras. \qed
\end{theorem}

\begin{lema}
There exists an algebraic group $\Ga$ and an injective homomorphism of algebraic
groups $\gamma: \Ga \to G$ such that $\OO(\Ga)$ is a Hopf subalgebra of $A$ and 
$A$ fits into the central exact sequence of Hopf algebras 
$\xymatrix{
1 \ar[r] & \OO(\Gamma) \ar[r]^{\hat{\iota}} &
A\ar[r]^{\hat{\pi}}
& H \ar[r] & 1}$,
where $H = A / \OO(\Ga)^{+}A$. Moreover, 
the following diagram is commutative
\begin{equation}\label{diag:diareci1}
\xymatrix{1 \ar[r]^{} & \Oc(G) \ar[r]^{\iota} \ar[d]_{^t\gamma} &
\Oephi \ar[r]^{\pi}
\ar[d]^{\kappa} & \qephi^{*} \ar[r]^{}\ar[d]^{} & 1\\
1 \ar[r] & \OO(\Gamma) \ar[r]^{\hat{\iota}} &
A\ar[r]^{\hat{\pi}}
& H \ar[r] & 1.}
\end{equation}
\end{lema}

\pf
Let $K=\kappa(\iota(\OO(G))) $. Since $\OO(G)$ is central in $\Oephi$, $K$ is central
in $A$ and there exists an algebraic group $\Ga$ and an algebraic group 
homomorphism $\gamma:\Ga \to G$ such that $K=\OO(\Ga)$ and 
$^{t}\gamma: \OO(G) \to \OO(\Ga)$ is the Hopf algebra epimorphism $\kappa\circ \iota|_{\OO(G)}$.
Moreover, if we set $H = A / K^{+}A$, then the sequence 
$\xymatrix{
1 \ar[r] & \OO(\Gamma) \ar[r]^{\hat{\iota}} &
A\ar[r]^{\hat{\pi}}
& H \ar[r] & 1}$ is exact and the diagram \eqref{diag:diareci1} is commutative.
\epf

By the lemma above, $H^{*}$ is a Hopf subalgebra of $\qephi$. Thus, by Lemma \ref{lema:subalgparadeuphi}
it is determined by a triple $( I_{+}, I_{-}, \Sigma^{\varphi})$. Let $P$ be the subgroup of $G$
determined by the pair $(I_+,I_-)$, $\liep = \Lie(P)$ and $\OephiP$,
$\qephip$ the quantum algebras given by
Proposition \ref{prop:paraL}. In particular, we have that 
$H^{*}\subseteq \qephip \subseteq \qephi$ and by Proposition \ref{prop:propdih},
$H \simeq \qephip^{*}/ \langle D^z-1 :\ D^{z} \in N^\varphi \rangle$, where
$N^{\varphi}$ is determined by $\Sigma^{\varphi}$ as in Remark \ref{rmk:notathenDi} $(d)$.
Denote by 
$\nu: \qephip \to H$ the corresponding epimorphism.

The next lemma follows from \cite[Lemma 3.1]{AG}, but adapted to the twisted case. 
\begin{lema} 
 The diagram \eqref{diag:diareci1} factorizes through the central exact sequence
$$\xymatrix{1 \ar[r]^{} &   \Oc(P)  \ar[r]^{\iota_P} &
\Oc_\epsilon^\varphi(P) \ar[r]^{\pi_P} & \qephip^{*} \ar[r]^ {} & 1.}$$
\end{lema}

\pf We want to show that $A$ fits into the commutative diagram
\begin{equation}
\xymatrix{1 \ar[r]^{} & \Oc(G) \ar[r]^{\iota} \ar[d]_{\res} &
\Oephi \ar[r]^{\pi}
\ar[d]^{\Res} & \qephi^{*} \ar[r]^{}\ar[d]^{\overline{\Res}} & 1\\
1 \ar[r]^{} &   \Oc(P)  \ar[r]^{\iota_P}\ar[d]_{^{t}\zeta} &
\Oc_\epsilon^\varphi(P) \ar[r]^{\pi_P}\ar[d]^{\psi} & \qephip^{*} \ar[r]^ {}\ar[d]^{\nu} & 1\\
1 \ar[r] & \OO(\Gamma) \ar[r]^{\hat{\iota}} &
A\ar[r]^{\hat{\pi}}
& H \ar[r] & 1.}
\end{equation}
To prove it, it suffices to show that $\Ker \Res \subseteq \Ker \kappa$. In order to do so, we realize
$\Oephi$ as a subalgebra of $\AA^{\varphi}_{\epsilon}= \AA^{\prime\prime}_{\varphi}\ot_{R}\QQ(\epsilon)$, 
see \cite[\S 3.6]{cv2}, Lemma  \ref{lem:mu}.

Let $\mu^{\varphi}_{\epsilon}: \Oephi \to \QUhatphie(\mathfrak{b}_-)^{\cop}\ot \QUhatphie(\mathfrak{b}_+)^{\cop}$ be 
the complexification of the injective algebra map $\mu^{\prime\prime}_{\varphi}$ given by \eqref{eq:defmu}. Then by Lemma \ref{lem:mu},
$\mu^{\varphi}_{\epsilon}(\Oephi) \subseteq \AA^{\varphi}_{\epsilon}$, which is the algebra generated by 
$f_\alpha^\varphi\otimes 1$,  $1\otimes e_\alpha^\varphi$ and 
$K_{-(1+\varphi)\lambda} \otimes K_{(1-\varphi)\lambda}$ for 
$\lambda \in P$ and $\alpha \in \Phi_+$. 

The proof follows by showing $\mu_\epsilon^\varphi(\Ker \Res)\subseteq\mu_\epsilon^\varphi(\Ker \kappa)$.
First note that 
$\mu_\epsilon^\varphi(\text{Ker }\text{Res})$ is the two-sided ideal $\II$ generated by 
$\{1\otimes e_k^\varphi, f_i^\varphi\otimes1|\ \alpha_k \notin I_-\text{, }\alpha_j\notin I_+\}$.
Indeed, by \cite[Proposition 2.7]{cv2}, Remark \ref{rmk:reci} $(a)$ and \eqref{eq:reci3teo} we have 
\begin{align*}
\mu_\epsilon^\varphi(\psi_{-\omega_i}^{-\alpha_i}\psi_{\omega_i}) & =
\left(\left(\epsilon^{-(\tau_i,\omega_i)}f^\varphi_{\alpha_i}\right)K_{-(1+\varphi)(\omega_i)}\otimes K_{(1-\varphi)(\omega_i)}\right)
\left(K_{(1+\varphi)(\omega_i)}\otimes K_{-(1-\varphi)(\omega_i)}\right)
\\
& = \epsilon^{-(\tau_i,\omega_i)}f^\varphi_{\alpha_i}\otimes 1.
\end{align*}
Analogously, we have $\mu_\epsilon^\varphi(\psi_{\omega_i}\psi_{-\omega_i}^{\alpha_k})=\epsilon^{-(\tau_i,\omega_i)}1\otimes e_{\alpha}^{\varphi}$. 
Since by definition  $\psi_{-\omega_i}^{\alpha_k}$, $\psi_{-\omega_i}^{-\alpha_k}\in \Ker \Res$ 
when $\alpha_k\notin I_-$, $\alpha_j\notin I_+$, we obtain that $1\otimes e_k^\varphi$, $f_i^\varphi\otimes1 \in \mu_{\epsilon}^{\varphi}(\Ker \Res)$
for $\alpha_k \notin I_-$ and $\alpha_j\notin I_+ $.
Conversely, assume $f\in \Ker \Res$. Then $f_{|\Gamma_\epsilon^\varphi(\p)}=0$ and 
$\langle \mu_\epsilon^\varphi(f),FM\otimes NE\rangle=f(FMNE)=0$ for all elements $FMNE$ in a basis of $\Gamma_\epsilon^\varphi(\p)$.
Using the perfect pairing \eqref{eq:pairphipi} on $\epsilon$, it follows that $\mu^{\varphi}_{\epsilon}(f) \subseteq \II$.

The proof that $\II \subseteq \mu^{\varphi}_{\epsilon}(\Ker \kappa)$ is analogous to the proof of 
\cite[Lemma 3.1]{AG}. 
\epf 

Note that the map $^{t}\zeta:\OO(P)\to \OO(\Ga)$ is given by the restriction $\psi|_{\OO(P)}$. 
Hence, $^{t}\zeta \res  =\ ^{t}\gamma$ and $\Img \gamma \subseteq P$.

We end the proof of Theorem \ref{thm:reciteo} with the following lemma.
Its proof is analogous to the case $\varphi=0$ and will be given without any detail.

\begin{lema}\cite[Lemmata 3.2 \& 3.3]{AG}
There exists a group homomorphism $\delta: N^{\varphi} \to \widehat{\Ga}$ such that 
the two-sided ideal  $J_{\delta}$ of $A^{\varphi}_{\epsilon, \liep,\gamma}$
generated by the elemens $\delta(D^{z}) - \partial^{z}$ for $D^{z}$
in $N^{\varphi}$ is a Hopf ideal,  
$A\simeq A_{\D^\varphi} = A^{\varphi}_{\epsilon, \liep,\gamma}/J_{\delta}$ as Hopf algebras and $A$
fits into the commutative diagram 
$$
\xymatrix{1 \ar[r]^{} & \Oc(G) \ar[r]^{\iota} \ar[d]_{\res} &
\Oephi \ar[r]^{\pi}
\ar[d]^{\Res} & \qephi^{*} \ar[r]^{}\ar[d]^{\overline{\Res}} & 1\\
1 \ar[r]^{} &   \Oc(P)  \ar[r]^{\iota_P}\ar[d]_{^t\gamma} &
\Oc_\epsilon^\varphi(P) \ar[r]^{\pi_P}\ar[d]^{\psi} & \qephip^{*} \ar[r]^{}\ar[d]^{\id} & 1\\
1 \ar[r]^{} & \OO(\Gamma) \ar[r]^{j} \ar[d]_{\id}& 
A^\varphi_{\epsilon,\mathfrak{p},\gamma}\ar[r]^{\overline{\pi}} \ar[d]^{}& \qephip^{*}\ar[r]^{}\ar[d]^{\nu} & 1\\
1 \ar[r] & \OO(\Gamma) \ar[r]^{\hat{\iota}} &
A\ar[r]^{\hat{\pi}}
& H \ar[r] & 1.}
$$ 
\end{lema}

\pf (\textit{Sketch}) Using that $A^\varphi_{\epsilon,\mathfrak{p},\gamma}$ is given by a pushout, one first shows that $A$ fits into 
the commutative diagram above. Then, using the commutativity of the 
diagram, one proves that there exists a group homomorphism $\delta: N^{\varphi} \to \widehat{\Ga}$ and 
a Hopf ideal $J_{\delta}$ such that $A \simeq A^{\varphi}_{\epsilon, \liep,\gamma}/J_{\delta}$.
\epf

\subsubsection{Isomorphism classes of quantum subgroups}
In this subsection we parametrize the Hopf algebra quotients of $\Oephi$ up to isomorphism.
To do so, we first define a partial order on the isomorphism classes of quotients of $\Oephi$ and on the
set of twisted subgroup data.

\smallbreak
Let $\q (\Oephi)$ be the category
whose objects are surjective Hopf algebra maps $\kappa: \Oephi \to A$. If $\kappa:
\Oephi \to A$ and $\kappa': \Oephi \to A'$ are such maps, then an arrow
$\xymatrix{\kappa\ar[0,1]^{\alpha}& \kappa'}$ in $\q (\Oephi)$ is a Hopf algebra
map $\alpha: A\to A'$ such that $\alpha \kappa = \kappa'$. In this language, a
\emph{quotient} of $\Oephi$ is just an isomorphism class of objects in
$\q(\Oephi)$; let $[\kappa]$ denote the class of the map $\kappa$. There is a
partial order in the set of quotients of $\Oephi$, given by $[\kappa]\leq
[\kappa']$ iff there exists an arrow $\xymatrix{\kappa\ar[0,1]^{\alpha}& \kappa'}$
in $\q(\Oephi)$. Note that $[\kappa]\leq [\kappa']$ and $[\kappa']\leq [\kappa]$
implies $[\kappa] = [\kappa']$.
Our goal is to describe the partial order in $\q(\Oephi)$.

Let $I_{\pm}, I_{\pm}^{\prime}\subseteq \pm\Pi$. If $I_{+}^{\prime} \subseteq I_{+}$ and $I_{-}^{\prime} \subseteq I_{-}$, 
then $I^{\prime} \subseteq I$ and
$\Tt^\varphi_{I'}\subseteq\Tt^\varphi_{I}$. Thus, there exists an epimorphism 
$\Tt^\varphi_{I'^c}\twoheadrightarrow \Tt^\varphi_{I^c}$ 
which induces a monomorphism $\eta: \widehat{\Tt^\varphi_{I^c}}\hookrightarrow \widehat{\Tt^\varphi_{I'^c}}$.

\begin{definition}
Let $\D^{\varphi}=(I_+,I_-,N^{\varphi},\Gamma,\gamma,\delta)$ and $\D^{\varphi\prime}=(I'_+,I'_-,N^{\varphi\prime},\Gamma',\gamma',\delta')$ be twisted 
subgroup data with respect to $\Oephi$. We say that $\D^{\varphi}\leq\D^{\varphi\prime}$ if and only if:
\begin{itemize}
\item[$\triangleright$] $I'_+\subseteq I_+$, $I'_-\subseteq I_-$.
\item[$\triangleright$] $\eta(N^{\varphi}) \subseteq  N^{\varphi\prime}$.
\item[$\triangleright$] there exists an algebraic group homomorphism $\tau: \Ga^{\prime} \to \Ga$ such that $\gamma\tau=\gamma'$.
\item[$\triangleright$] $\delta'\eta=\tau^t\delta$.
\end{itemize}
\end{definition}
Moreover, we say that $\D^{\varphi}\sim\D^{\varphi\prime}$ if and only if 
$\D^{\varphi}\leq\D^{\varphi\prime}$ and $\D^{\varphi}\leq\D^{\varphi\prime}$. In particular, this implies that 
 $I'_+= I_+$, $I'_-= I_-$, $N^{\varphi} = N^{\varphi\prime}$, $\tau$ is an isomorphism
 and $\delta'=\tau^t\delta$.

Our last theorem yields the parametrization of the quotients of $\Oephi$ up to isomorphism. The proof 
is analogous to the case $\varphi=0$ since it relies on the commutativity of the diagram \eqref{diag:pushout} and 
general constructions of the sucesive quotients. For these reasons, it will be omitted. See \cite[Theorem 2.20]{AG}.

\begin{theorem}\label{thm:classification}
Let $\D^{\varphi}$ and $\D^{\varphi\prime}$ be twisted subgroup data and 
$\kappa: \Oephi \to A_{\D^{\varphi}}$, $\kappa': \Oephi \to A_{\D^{\varphi'}}$ the corresponding quotients.
Then $[\kappa]\leq[\kappa']$ 
if only if $\D^{\varphi}\leq\D^{\varphi\prime}$. \qed
\end{theorem}

\subsubsection{Properties of the quotients}\label{subsec:propquot}
We end the paper with a list of some properties of the quotients. Apart from item $(v)$, the proof is analogous to \cite[Proposition 3.8]{AG2}.
\begin{prop}
Let $D^{\varphi} = (I_{+}, I_{-}, N^{\varphi}, \Ga,\gamma, \delta)$ be a twisted subgroup datum. 
\begin{enumerate}
 \item[$(i)$] If $A_{D^{\varphi}}$ is pointed, then $I_{+} \cap I_{-} =\emptyset $ and 
 $\Ga$ is a subgroup of the group of upper triangular matrices of some
size. In particular, if $\Ga$ is finite, then it is abelian.
 \item[$(ii)$] $A_{D^{\varphi}}$ is semisimple if and only if $I_{+} \cup I_{-} =\emptyset$ and $\Ga$ is finite.
 \item[$(iii)$] If $\dim A_{D^{\varphi}} <\infty$ and $A_{D^{\varphi}}^{*}$ is pointed, then $\gamma(\Ga )$ is included 
 in the fixed torus of $G$.
 \item[$(iv)$] If $A_{D^{\varphi}}$ is co-Frobenius then $\Ga$ is reductive.
 \item[$(v)$] If $\varphi$ and $(I_{+}, I_{-},\Sigma^{\varphi})$ are such that $\Sigma^{\varphi}=\Tt^{\varphi}$ but $\Sigma\neq \Tt$, then 
 $ A_{\D^{\varphi}}$ is not a $2$-cocycle deformation of $A_{\D}$.
\end{enumerate}
\end{prop}

\pf
We prove only $(v)$. If $\varphi$ and $(I_{+}, I_{-},\Sigma^{\varphi})$ are 
such that $\Sigma^{\varphi}=\Tt^{\varphi}$ but $\Sigma\neq \Tt$, then 
$N^{\varphi}=1$ and $N\neq 1$. Then, the quotient  
$H^{\varphi}= \qephip^\ast/\langle D^z-1 :\ D^{z} \in N^\varphi \rangle$ cannot be a $2$-cocycle deformation
of $H=\qep^\ast/\langle D^z-1 :\ D^{z} \in N \rangle$ since they have different dimension. 
If $ A_{D^{\varphi}}$ were a $2$-cocycle deformation of $A_{D}$, then 
by a chasing diagram argument we would have that $H^{\varphi}$ is a $2$-cocycle 
deformation of $H$, a contradiction, see Example \ref{exa:2}. 
\epf

\end{document}